\RequirePackage{fix-cm}
\documentclass[smallextended]{svjour3}       
\smartqed  
\usepackage{mathptmx}      
\usepackage{amsmath,amssymb}
\usepackage{mathtools}	
\usepackage{bm}
\usepackage[boxed,linesnumbered]{algorithm2e}
\usepackage{graphicx} \graphicspath{ {figures/} }
\usepackage{subcaption}
\usepackage[usenames, dvipsnames]{color}  
\usepackage{enumerate}
\usepackage[numbers]{natbib}
\def \b  {\mathbf{b}}

\def \g  {\mathbf{g}}
\def \h  {\mathbf{h}}

\def \s  {\mathbf{s}}
\def \t  {\textbf{t}}

\def \v  {\mathbf{v}}

\def \x  {\textbf{x}}
\def \bx {\bar{\textbf{x}}}

\def \z  {\textbf{z}}

\def \bA {\mathbf{A}}
\def \bB {\mathbf{B}}
\def \bC {\mathbf{C}}

\def \bG {\mathbf{G}}
\def \bH {\mathbf{H}}
\def \bI {\mathbf{I}}

\def \bO {\mathbf{O}}
\def \bP {\mathbf{P}}
\def \bR {\mathbf{R}}
\def \bS {\mathbf{S}}

\def \bW {\mathbf{W}}
\def \bX {\mathbf{X}}
\def \bY {\mathbf{Y}}

\def \cB {\mathcal{B}}

\def \cL {\mathcal{L}}

\def \cS {\mathcal{S}}
\def \cT {\mathcal{T}}

\def \O  {\mathbb{O}}
\def \R  {\mathbb{R}}
\def \S  {\mathbb{S}}

\def \Del {\bm{\Delta}}

\def \Lam {\bm{\Lambda}}

\def \zer {\mathbf{0}}

\DeclarePairedDelimiter \norm {\lVert}{\rVert}
\DeclarePairedDelimiter \dotp {\big\langle}{\big\rangle}

\makeatletter
\newcommand{\leqnomode}{\tagsleft@true\let\veqno\@@leqno}
\newcommand{\reqnomode}{\tagsleft@false\let\veqno\@@eqno}
\makeatother

\journalname{}
\begin{document}

\title{Convergence Analysis of Nonconvex ADMM for Rigid Registration 	\thanks{K.~N.~Chaudhury was supported by a MATRICS Grant from DST-SERB, Government of India.}
}


\author{Aditya~V.~Singh         \and
        Kunal~N.~Chaudhury 
}


\institute{Aditya~V.~Singh \at
              Indian Institute of Science\\
              \email{adityavs@iisc.ac.in}           
           \and
           Kunal~N.~Chaudhury \at
              Indian Institute of Science \\
              \email{kunal@iisc.ac.in} 
}

\date{Received: date / Accepted: date}

\maketitle

\begin{abstract}
We consider the problem of rigid registration, where we wish to jointly register multiple point sets via rigid transforms. This arises in applications such as sensor network localization, multiview registration, and protein structure determination. The least-squares estimator for this problem can be reduced to a rank-constrained semidefinite program (\ref{sdp}). It was recently shown that by formally applying the alternating direction method of multipliers (ADMM), we can derive an iterative solver (\ref{ncadmm}) for \ref{sdp}, wherein each subproblem admits a simple closed-form solution. The empirical success of \ref{ncadmm} has been demonstrated for multiview registration. However, its convergence does not follow from the existing literature on nonconvex ADMM. In this work, we study the  convergence of \ref{ncadmm} and our main findings are as follows. We prove that any fixed point of \ref{ncadmm} is a stationary (KKT) point of \ref{sdp}. Moreover, for clean measurements, we give an explicit formula for the ADMM parameter $\rho$, for which \ref{ncadmm} is guaranteed to converge to the global optimum (with arbitrary initialization). If the noise is low, we can still show that the iterates converge to the global optimum, provided they are initialized sufficiently close to the optimum. On the other hand, if the noise is high, we explain why \ref{ncadmm} becomes unstable if $\rho$ is less than some threshold, irrespective of the initialization. We present simulation results to support our theoretical predictions. The novelty of our analysis lies in the fact that we exploit the notion of tightness of convex relaxation to arrive at our convergence results.
\keywords{Semidefinite program \and Rank constraint \and Nonconvex \and ADMM \and Convergence}
\end{abstract}

\section{Introduction} \label{intro}
In rigid registration, we want to infer the global coordinates of a set of points, given the coordinates of overlapping subsets of these points in different local coordinate systems \cite{knc}. The local coordinate systems are related to each other by rigid transforms, which are otherwise unknown. This problem arises in applications such as sensor network localization, multiview registration, protein structure determination, and manifold learning \cite{sanyal,mihai12,arap10,krishnan,sharp2004,mihaimol12,fang2013,zhang2004}, where we wish to reconstruct an underlying global structure based on observations of multiple local sub-structures. 
For instance, consider an adhoc wireless network consisting of geographically distributed sensor nodes with limited radio range. To make sense of the data collected from the sensors, one usually requires the positions of the individual sensors. The positions can be found simply by attaching a GPS with each sensor, but this is often not feasible due to cost, power, and weight considerations. On the other hand, we can estimate (using time-of-arrival) the distance between sensors that are within the radio range of each other \cite{mao2007}. The problem of estimating sensor locations from the available inter-sensor distances is referred to as sensor network localization (SNL) \cite{mao2007,shang2004}. 
Efficient methods for accurately localizing small-to-moderate sized networks have been proposed over the years \cite{soares15,simonetto14,wang08,biswas06}. However, these methods typically cannot be used to localize large networks. To address this, scalable divide-and-conquer approaches for SNL have been proposed \cite{arap10,mihai12,knc2015,sanyal}, where the large network is first subdivided into smaller subnetworks which can be efficiently and accurately localized (pictured in Fig. \ref{fig:regscen}(a)). 
Each subnetwork (called patch) is localized independent of other subnetworks. Thus, the coordinates returned for a patch will in general be an arbitrarily rotated, flipped, and translated version of the ground-truth coordinates (Fig. \ref{fig:regscen}(b)). Consequently, the network is divided into multiple patches, where each patch can be regarded as constituting a local coordinate system which is related to the global coordinate system by an unknown rigid transform. We now want to assign coordinates to all the nodes in a global coordinate system based on these patch-specific local coordinates. 

Rigid registration also comes up in multiview registration, where the objective is to reconstruct a $3$D model of an object based on partial overlapping scans of the object. Here, the scans can be seen as patches, which are to be registered in a global reference frame via rotations and translations. Similar situation arises in protein conformation, where we are required to determine the $3$D structure of a protein (or other macromolecule) from overlapping fragments \cite{mihaimol12,fang2013}.

\begin{figure*}[h]
	\centering
	\includegraphics[width=1\linewidth]{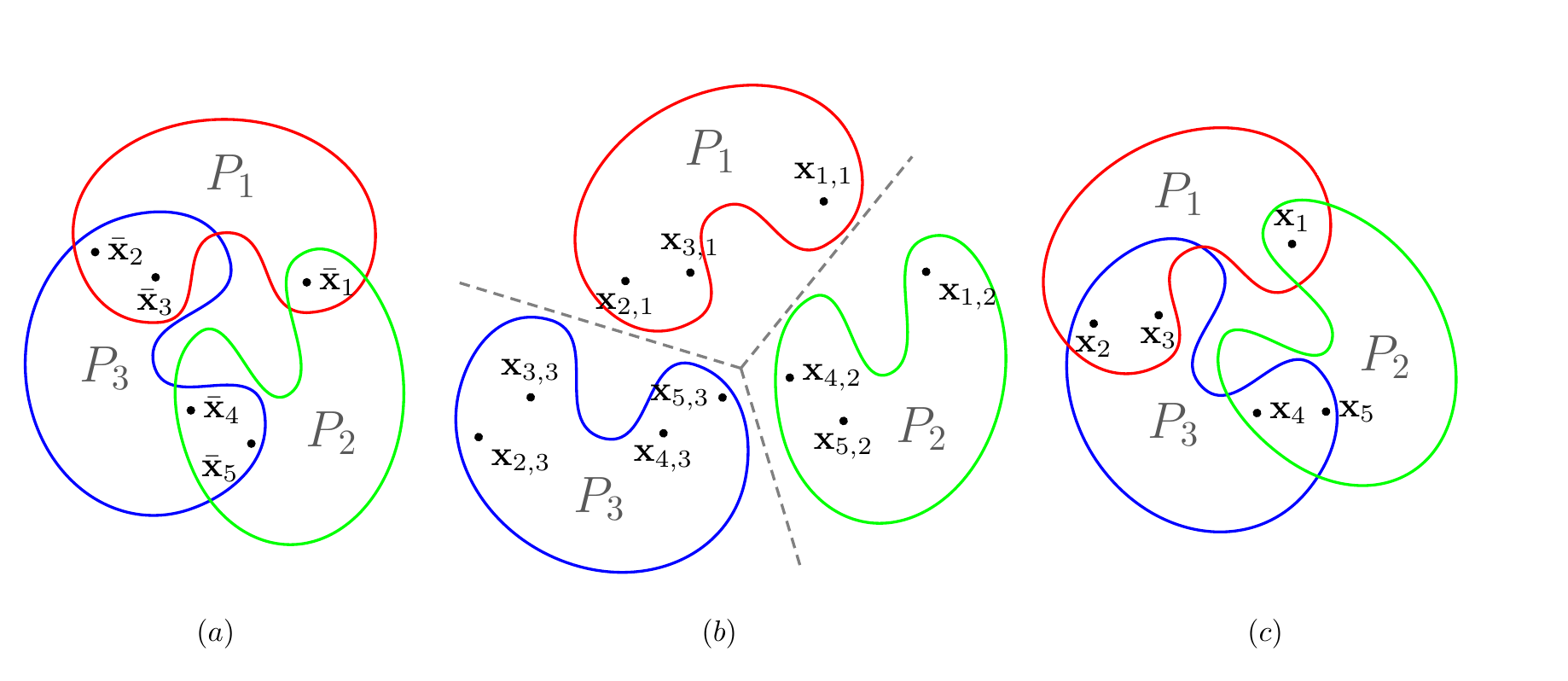}
	\caption
	{
		A simple registration scenario. $(a)$ Ground truth network, where $\bx_1,\ldots,\bx_5$ are the true global coordinates. $(b)$ Three local coordinate systems (patches), where $\x_{k,i}$ is the coordinate of the $k$-th node in the patch $P_i$. $(c)$ Registered network. Note that the registered and the ground-truth networks are related by a global rigid transform, which is the best we can do with the given data. 
	}
	\label{fig:regscen}
\end{figure*}


\subsection{Problem Statement} \label{subsec:probstat}
We now formally describe the rigid registration problem and set up the rank constrained semidefinite program (\ref{sdp}) that is the focus of our analysis. We first clarify our notation: $[m:n]$ denotes the set of integers $\{m,\ldots,n\}$; $\O(d)$ denotes the set of $d \times d$ orthogonal matrices; $\S^m$ denotes the set of symmetric matrices of size $m \times m$; $\S^m_+$ denotes the set of symmetric positive semidefinite matrices of size $m \times m$; $\bI_d$ denotes the $d \times d$ identity matrix. We will regard a matrix $\bA \in \R^{Md \times Md}$ as a block matrix composed of $M \times M$ blocks, where each block is of size $d \times d$; we will refer to the $(i,j)$-th block as $[A]_{ij}$, where $1 \leq i,j \leq M$.
We also recall that a rigid transform in $\R^d$ is a composition of an orthogonal transform (rotation, reflection) and a translation, and is denoted by $(\bO,\t)$, where $\bO \in \O(d)$ denotes the orthogonal transform, and $\t \in \R^d$ denotes the translation. Specifically, a rigid transform $(\bO,\t)$ maps a point $\x$ to the point $\bO \x + \t$.
 
Suppose a network consists of $N$ nodes in $\R^d$, which we label using $\cS = [1 : N]$. The network is sub-divided into $M$ patches (subsets of $\cS$), labeled $P_1,\cdots,P_M$. Let $\bar{\x}_1,\ldots,\bar{\x}_N \in  \R^d$ be the true coordinates of the $N$ nodes in some global coordinate system. We associate with each patch a local coordinate system: If node $k$ belongs to $P_i$, let $\x_{k,i} \in \R^d$ be the local coordinate of node $k$ in patch $P_i$. In other words, if $(\bar{\bO}_i,\bar{\t}_i)$ is the (unknown) rigid transform (defined with respect to the global coordinate system) associated with patch $P_i$, then we have
\begin{equation}
\label{exact}
\bar{\x}_k =  \bar{\bO}_i \x_{k,i} + \bar{\t}_i.
\end{equation}

In general, the local coordinate measurements are noisy and we cannot expect \eqref{exact} to hold exactly. Instead, we resort to least-squares fitting to solve the registration problem. More precisely, given the local coordinate measurements $(\x_{k,i})$, we consider the following least-squares minimization to estimate the rigid transforms and the global coordinates \cite{knc}:  
\leqnomode
\begin{equation} 
\label{reg} 
\tag{REG}
\underset{ (\bO_i), (\t_i), (\z_k) }{\min} \ \sum_{i=1}^{M} \sum_{k \in P_i} \norm{\z_k -  \left( \bO_i \x_{k,i} + \t_i \right) }^2.
\end{equation}
\reqnomode

The fundamental difficulty of \ref{reg} stems from the fact that variables $\bO_i$ are constrained to be in $\O(d)$, which is a nonconvex and disconnected set \cite{absil2009optimization}. For instance, $\O(1) = \{-1,1\}$, and $\O(2)$ is topologically equivalent to the union of two disjoint circles in $\R^2$. The disconnectedness makes it difficult to apply local optimization methods; in particular, if one initializes the method on the wrong component, one cannot hope to get close to the global optimum.

To isolate the difficulty in \ref{reg}, we will now reformulate it in a form that involves only the orthogonal transforms. To this end, note that \ref{reg} can be rewritten as
\begin{equation} \label{rewrite}
\underset{(\bO_i)}{\min} \ \left[ \underset{(\t_i), (\z_k)}{\min} \ \sum_{i=1}^{M} \sum_{k \in P_i} \norm{\z_k -  (\bO_i \x_{k,i} + \t_i) }^2 \right].
\end{equation}
It was observed in \cite{knc} that if we fix the orthogonal transforms $(\bO_i)$, the minimization problem inside the square brackets  in \eqref{rewrite} is a quadratic convex optimization in $(\t_i)$ and $(\z_k)$, with a closed-form optimum which is linear in $(\bO_i)$. In other words, if we knew the optimal orthogonal transforms for \ref{reg}, the optimal values of $(\t_i)$ and $(\z_k)$ could be computed using a simple linear transform. Based on this observation (we refer the reader to \cite{knc} for further details), the variables $(\t_i)$ and $(\z_k)$ can be eliminated, and \ref{reg} is reduced to the following optimization involving only the orthogonal transforms:
\leqnomode
\begin{equation} \label{orthreg} 
\tag{O-REG}
\underset{\bO_1,\ldots,\bO_M \in \O(d)}{\min} \ \sum_{i,j=1}^M \mathrm{Tr} \left(\bC_{ij} \bO_j^{\top} \bO_i \right).
\end{equation}
\reqnomode
Here, $\bC_{ij} \in \R^{d \times d}$, and $\bC_{ji} = \bC_{ij}^\top$. More specifically, \ref{orthreg} is equivalent to \ref{reg} in the following sense: If $(\bO_i^\ast)$, $(\t_i^\ast)$, $(\z_k^\ast)$ are globally optimal for \ref{reg}, then
\begin{equation} \label{equiv}
\begin{aligned}
\sum_{i=1}^{M} \sum_{k \in P_i} \norm{\z_k^\ast -  (\bO_i^\ast \x_{k,i} + \t_i^\ast) }^2
&= \underset{\bO_1,\ldots,\bO_M \in \O(d)}{\min} \ \sum_{i,j=1}^M \mathrm{Tr} \left(\bC_{ij} \bO_j^{\top} \bO_i \right) \\
&= \mathrm{Tr} \left(\bC_{ij} \bO_j^{\ast \top} \bO_i^\ast \right).
\end{aligned}
\end{equation}

Note that the objective in \ref{orthreg} can be regarded as a quadratic form in $(\bO_i)$. That is, defining $\bO = \left( \bO_1 \cdots \bO_M \right) \in \R^{d \times Md}$ and $\bC = \left( (\bC_{ij})_{1 \leq i,j \leq M} \right) \in \R^{Md \times Md}$, we have
\begin{equation} \label{quadform}
\sum_{i,j=1}^M \mathrm{Tr} \left(\bC_{ij} \bO_j^{\top} \bO_i \right) =  \mathrm{Tr} \left( \bC \bO^\top \bO \right).
\end{equation}
Here, matrix $\bC$ is symmetric, which follows from the fact that $\bC_{ji} = \bC_{ij}^\top$. Furthermore, $\bC$ is positive semidefinite \cite{knc}; this is because the objective in \ref{orthreg} is simply  a rewriting of the objective in \ref{reg}, which is always nonnegative, being the sum of norms. 

We will now formulate \ref{orthreg} as a rank-constrained semidefinite program. To this end, define Gram matrix $\bG = \bO^\top \bO \in \S^{Md}_+$. Note that the $(i,j)$-th block of $\bG$ is given by $[\bG]_{ij} = \bO_i^\top \bO_j$. In particular, the diagonal blocks of $\bG$ are the identity matrix, since $[\bG]_{ii} = \bO_i^\top \bO_i = \bI_d$. Following this, \ref{orthreg} can be shown \cite{knc} to be equivalent to
\leqnomode
\begin{equation} \label{sdp} \tag{REG-SDP}
\begin{aligned}
& \min_{\bG \in \S^{Md}_+} && \mathrm{Tr}\left( \bC \bG \right) \\
& \text{subject to}        && [\bG]_{ii} = \bI_d, \quad i \in  [1:M], \\     
&						   && \mathrm{rank}(\bG) = d.
\end{aligned}
\end{equation}
\reqnomode

\ref{sdp} is a standard semidefinite program \cite{aliza}, but with an additional rank constraint, which makes it nonconvex. An apparent drawback in reformulating \ref{orthreg} as \ref{sdp} is that the number of variables of optimization have increased from $d \times Md$ to $Md \times Md$. On the plus side, though, the domain of optimization in \ref{sdp} is connected, which means that we can expect local optimization methods to converge to global optimum of \ref{sdp}. Moreover, the nonconvexity of the problem is isolated in the rank constraint (i.e. dropping the rank constraint from \ref{sdp} makes the problem convex), a fact we will exploit in our analysis.


\subsection{Algorithm} 
In this paper, we theoretically analyze the iterative algorithm proposed in \cite{miraj} for solving the nonconvex program \ref{sdp}. This algorithm is based on ADMM (alternating direction method of multipliers), a method traditionally employed for convex programs, where it has strong convergence guarantees \cite{boyd}. Recently, however, there have been a spate of works where applying ADMM to nonconvex programs have yielded impressive results \cite{boyd,chartrand2013,diamond,hong,miraj}. But the theory of convergence for nonconvex ADMM has not yet caught up with its convex counterpart. In particular, as we shall see when we discuss related works in Section \ref{subsec:relwork}, existing results on nonconvex ADMM do not apply to our algorithm. 

To see how we can apply ADMM to \ref{sdp}, we first reformulate \ref{sdp} as
\begin{equation} \label{eqn:ncadmm}
\begin{aligned}
& \underset{\bG, \bH \in \S^{Md}}{\min} && \mathrm{Tr}(\bC \bG) \\
& \text{subject to} && \bG \in \Omega, \: \bH \in \Theta, \\
&					&& \bG - \bH = \zer,
\end{aligned}
\end{equation}
where, $\Omega = \{\bX \in \S^{Md}_+ : \mathrm{rank}(\bX) \leq d \}$, and $\Theta = \{\bX \in \S^{Md} : [\bX]_{ii} = \bI_d \}$. The equivalence of \ref{eqn:ncadmm} to \ref{sdp} follows from the fact that the rank of any $\bH \in \Theta$ is at least $d$, since its diagonal blocks are $\mathbf{I}_d$. Moreover, if $\bH$ is feasible for \eqref{eqn:ncadmm}, the rank of $\bH$ cannot be more than $d$, since $\bH =\bG$ and  $\bG \in \Omega$. The splitting of the constraint into $\Omega$ and $\Theta$, with an additional linear (consensus) constraint, makes the problem algorithmically amenable to ADMM \cite{boyd}. For some fixed $\rho > 0$, the augmented Lagrangian for \eqref{eqn:ncadmm} is defined to be
\begin{equation} \label{eqn:auglang}
\cL_\rho (\bG,\bH,\Lam) = \mathrm{Tr}\left( \bC \bG \right) 
+ \mathrm{Tr}( \Lam ( \bG - \bH) )
+ \frac{\rho}{2} \norm{\bG-\bH}^2,
\end{equation}
where, $\Lam \in \S^{Md}$ is the dual variable for the constraint $\bG-\bH = \zer$ \cite{boyd}. The ADMM algorithm in \cite{miraj}, initialized with some $\bH^0$ and $\Lam^0$, involves the following updates for $k \geq 0$:
\leqnomode
\begin{equation}  \label{ncadmm} \tag{REG-ADMM}
\begin{aligned}
\bG^{k+1} &= \: \underset{\bG \in \Omega}{\text{argmin}} \:\: 
\cL_\rho \left( \bG, \bH^k, \Lam^k \right); \\
\bH^{k+1} &= \: \underset{\bH \in \Theta}{\text{argmin}} \:\: 
\cL_\rho \left( \bG^{k+1}, \bH, \Lam^k \right); \\
\Lam^{k+1} &= \: \Lam^k + \rho \left( \bG^{k+1} - \bH^{k+1} \right). 
\end{aligned}
\end{equation}
\reqnomode
As observed in \cite{miraj}, the first two sub-problems in \ref{ncadmm} ($\bG$-update and $\bH$-update steps) can be expressed as matrix projections onto $\Theta$ and $\Omega$ respectively. Namely,
\begin{equation} \label{ncadmmproj}
\begin{aligned}
\bG^{k+1} &= \Pi_\Omega \left( \bH^k - \rho^{-1} \left( \bC + \Lam^k \right) \right), \\
\bH^{k+1} &= \Pi_\Theta \left( \bG^{k+1} + \rho^{-1} \Lam^k \right),
\end{aligned}
\end{equation}
where, $\Pi_\Omega(\cdot)$ and $\Pi_\Theta(\cdot)$ denotes projection onto $\Omega$ and $\Theta$. Importantly, these projections are computationally efficient.
More specifically, projection on $\Theta$ can be computed simply by setting the diagonal blocks of the input matrix to $\bI_d$. On the other hand, projection on $\Omega$ can be computed efficiently by computing its top $d$ eigenvalues and retaining the ones that are positive \cite{miraj}. That is, if $\bA= \lambda_1 \boldsymbol{u}_1 \boldsymbol{u}_1 ^{\top}+\cdots+ \lambda_{Md} \boldsymbol{u}_{Md} \boldsymbol{u}_{Md} ^{\top}$, where $\lambda_1 \geq  \cdots \geq  \lambda_{Md}$ are the eigenvalues, and $\boldsymbol{u}_1,\ldots, \boldsymbol{u}_{Md}$ are the corresponding (orthonormal) eigenvectors, then
\begin{equation*}
\Pi_{\Omega}(\bA) = \sum_{i=1}^d \max(\lambda_i,0) \boldsymbol{u}_i \boldsymbol{u}_i ^{\top}.
\end{equation*}
Empirically, we find that \ref{ncadmm} has good convergence properties. In fact, it was shown in \cite{miraj} that its performance is comparable to state-of-the-art methods for the registration of three-dimensional multiview scans. Our aim in this paper is to study theoretical convergence of \ref{ncadmm}.


\subsection{Convex Relaxation of \ref{sdp}} \label{subsec:cadmm}
In our analysis of \ref{ncadmm}, we will be leveraging convergence properties of a closely related ADMM algorithm that solves a \emph{relaxed} (convex) version of \ref{sdp}. Observe that the nonconvex nature of the rigid registration problem \ref{reg} is isolated in the rank constraint in \ref{sdp}. As a result, by simply dropping the rank constraint, we obtain the following convex program, which we call the \emph{convex relaxation} of \ref{sdp}:
\leqnomode
\begin{equation} \label{csdp} \tag{C-SDP}
\begin{aligned}
& \min_{\bG \in \S^{Md}_+} && \mathrm{Tr}(\bC \bG)\\
& \text{subject to}      && [\bG]_{ii} = \bI_d, \quad i \in  [1:M].
\end{aligned}
\end{equation}
\reqnomode

An ADMM algorithm to solve \ref{csdp} was proposed in \cite{kncspl}. Initialized with some $\bH^0$ and $\Lam^0$, this involves the following updates for $k \geq 0$:
\leqnomode
\begin{equation}  
\label{cadmm} \tag{C-ADMM}
\begin{aligned}
\bG^{k+1} &= \: \underset{\bG \in \S^{Md}_+}{\text{argmin}} \:\: 
\cL_\rho \left( \bG, \bH^k, \Lam^k \right); \\
\bH^{k+1} &= \: \underset{\bH \in \Theta}{\text{argmin}} \:\: 
\cL_\rho \left( \bG^{k+1}, \bH, \Lam^k \right); \\
\Lam^{k+1} &= \: \Lam^k + \rho \left( \bG^{k+1} - \bH^{k+1} \right). 
\end{aligned}
\end{equation}
\reqnomode
The expression for the augmented Lagrangian $\cL_\rho$ in \ref{cadmm} is identical to that in \eqref{eqn:auglang}. Note that \ref{cadmm} is similar to \ref{ncadmm}, with the only difference being the $\bG$-update step. Namely, the minimization  is over the nonconvex set $\Omega$ in \ref{sdp}, while it is over the closed convex set $\S^{Md}_+$ in \ref{csdp}. This difference turns out to be crucial for convergence: \ref{cadmm} converges to the global optimum of \ref{csdp} with arbitrary initializations \cite{kncspl}. 

While \ref{cadmm} enjoys strong convergence guarantees, it has the typical drawbacks of a convex relaxation. First, the rank of the global optimum of \ref{csdp} is not guaranteed to be $d$, i.e., it might not even be feasible for \ref{sdp}. If the rank is greater than $d$, we have to ``round'' the   solution of \ref{csdp}  to a rank-$d$ matrix \cite{knc}, which will generally be suboptimal for \ref{sdp}.
Second, because the $\bG$-update requires us to optimize over the entire PSD cone $\S^{Md}_+$, \ref{cadmm} requires the full eigendecomposition of an $Md \times Md$ matrix at every iteration \cite{kncspl}. \ref{ncadmm} overcomes these drawbacks in the following manner. First, \ref{ncadmm} works directly with the nonconvex problem \ref{sdp}, and thus obviates the need for any rounding. Second, \ref{ncadmm} requires only the top $d$ eigenvectors in the $\bG$-update step, resulting in appreciable computational savings. However, these benefits come at a cost---it is usually difficult to derive theoretical guarantees for nonconvex optimization. As far as we know, convergence guarantees for \ref{ncadmm} do not follow from existing results on nonconvex ADMM (cf. Section \ref{subsec:relwork} for further details).

\begin{figure}
	\centering
	\begin{subfigure}[b]{\linewidth}
		\centering
		\includegraphics[width=\linewidth]{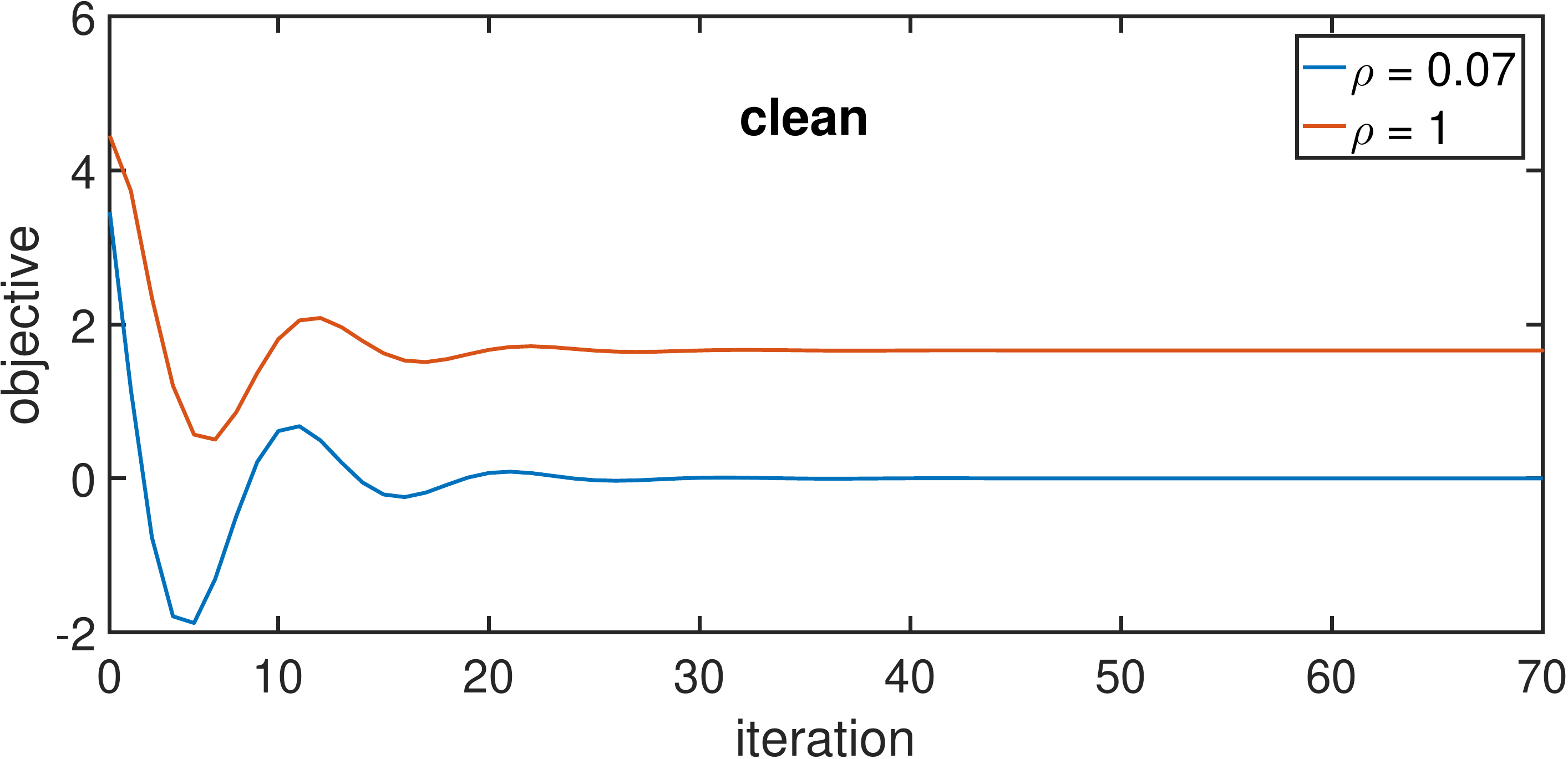}
		\caption
		{
			Even with clean measurements, the iterates get stuck in a local minimum when $\rho = 1$ (the optimum value is zero in this case).\\
		}
		\label{fig:cleanrho}
	\end{subfigure}
	\begin{subfigure}[b]{\linewidth}
		\centering
		\includegraphics[width=\linewidth]{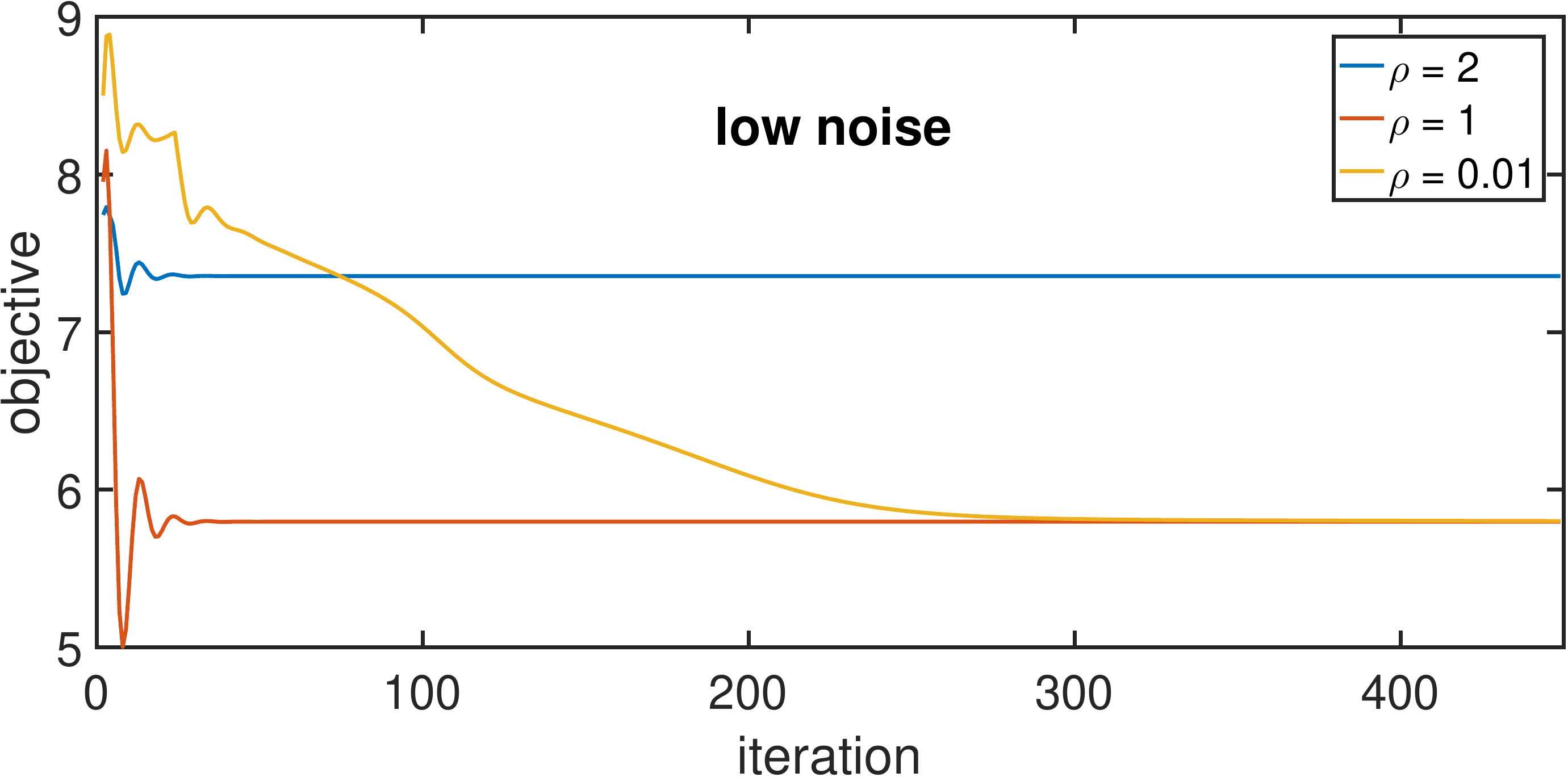}
		\caption
		{
			There are no oscillations for small $\rho$ when the noise is low.\\~\\
		}
		\label{fig:lonoiserho}
	\end{subfigure}	
	\begin{subfigure}[b]{\linewidth}
		\centering
		\includegraphics[width=\linewidth]{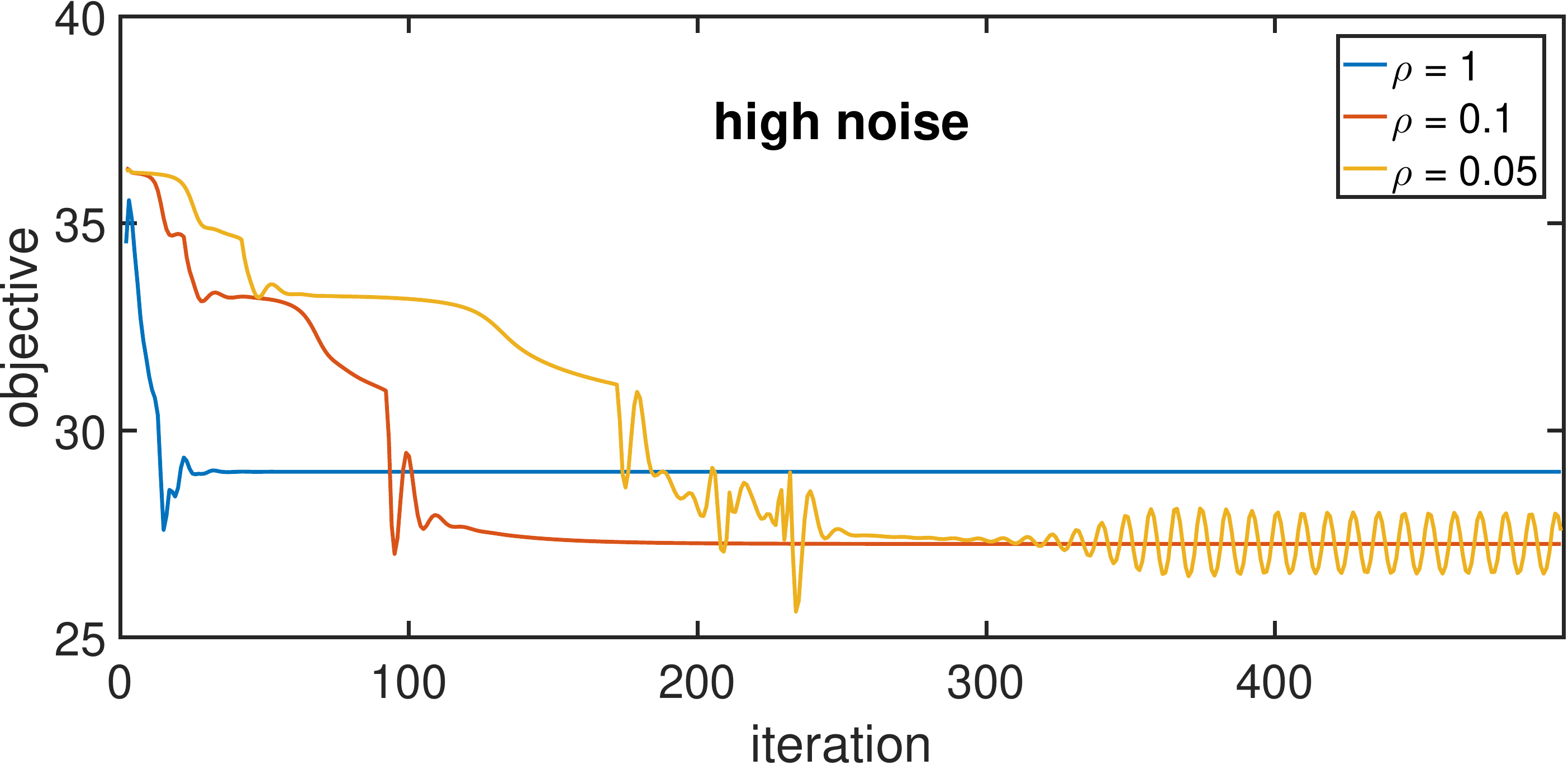}
		\caption
		{
			The iterates oscillate for small $\rho$ when the noise is high.
		}
		\label{fig:hinoiserho}
	\end{subfigure}
	\caption
	{
		Convergence behavior of \ref{ncadmm} at different noise levels in the data matrix $\bC$ (see the main text for description of the experiment).
	}
	\label{fig:noiserho}
\end{figure}


\subsection{Numerical Experiments} \label{subsec:numexp}
To understand the challenges involved in the convergence analysis of \ref{ncadmm}, we look at some simulation results for the sensor network localization problem. We consider a  two-dimensional network with ten nodes. There are three patches, where each patch contains all the ten nodes. The local coordinate system for each patch is obtained by arbitrarily rotating the global coordinate system. Moreover, we perturb the local coordinates using iid Gaussian noise. Our goal is to estimate the patch rotations up to a global transform. We set up  \ref{sdp} for this problem and solve it using \ref{ncadmm}. Fig. \ref{fig:noiserho} shows the dependence of \ref{ncadmm} on $\rho$ at different noise levels, where by ``noise level'' we mean the variance of the Gaussian noise. Notice that even when the local coordinates measurements are clean, \ref{ncadmm} may get stuck in a local minimum depending on $\rho$. This dependence of the limit point on $\rho$ is observed both at low and high noise levels. This is in contrast with \ref{cadmm}, where the iterates converge to a global minimum for any positive $\rho$ \cite{kncspl}. Also observe that when the noise is relatively large (Fig.  \ref{fig:hinoiserho}), the iterates of the algorithm may oscillate without converging if $\rho$ is small. Such non-attenuating oscillations are \emph{not} observed when the noise is low (see Fig.  \ref{fig:lonoiserho}).


\subsection{Contribution} \label{subsec:contri}
The foregoing simulation results provide the main motivation for this paper, namely, we wish to theoretically justify the observed behavior of \ref{ncadmm} in different noise regimes. We first clarify what we mean by ``noise''. As is clear from our discussion in Section \ref{subsec:probstat}, the data matrix $\mathbf{C}$ in \ref{sdp} ultimately depends on the local coordinate measurements $\x_{k,i}$, $i \in [1:M]$, $k \in P_i$. If these measurements are exact, we say that $\mathbf{C}$ is \emph{clean}; otherwise, we say that $\mathbf{C}$ is noisy. Furthermore, we make a distinction between \emph{low} and \emph{high} noise. In this context, we bring in the notion of \emph{tightness} of the convex relaxation \ref{csdp}. Recall that \ref{csdp} is derived by dropping the rank constraint. Let $\bG^\ast$ denote a global optimum of \ref{csdp}. If $\mathrm{rank}(\bG^\ast) = d$, then clearly $\bG^\ast$ is global optimum of \ref{sdp} as well. That is, global minimizer of the convex relaxation is a global minimizer for the original nonconvex program too. In this case, we say that the relaxation is \emph{tight} \cite{bandeira}. Empirically, we notice the well-known phenomena of phase transition for convex relaxations (e.g. see \cite{mixon2015phase,donoho2011}), where below a certain noise threshold, the relaxation remains tight (see Fig. \ref{fig:phasetran}). If the noise level is below (resp. above) this threshold, we say that the noise is low (resp. high).
An informal account of our main findings is as follows:
\begin{enumerate}
	\item We prove that any fixed point of \ref{ncadmm} is a stationary (KKT) point of \ref{orthreg} (Theorem \ref{thm:fixpt}).
	
	\item We rigorously establish the existence of low noise regime. That is, we show that there is a noise threshold below which the convex relaxation \ref{csdp} is tight (Theorem \ref{thm:tight}). 
	
	\item At low noise, we show that the \ref{ncadmm}  iterates converge to the global optimum, provided they are initialized sufficiently close to the optimum (Theorem \ref{thm:lownoise}).
	
	\item If the data matrix is clean, then, for arbitrary  initialization, we compute $\rho$ for which \ref{ncadmm} converges to the global optimum of \ref{sdp} (Theorem \ref{thm:cleancase}). 
	
	\item At high noise, we give a duality-based explanation of why the iterates exhibit non-attenuating oscillations when $\rho$ is small, and why no such oscillations are observed at low noise.
\end{enumerate}

The novelty of our analysis lies in the fact that we exploit the phenomenon of tightness of convex relaxation to prove convergence of the nonconvex ADMM. We contrast this approach with existing works on nonconvex ADMM (which rely on assumptions that do not apply to \ref{ncadmm}) in the following subsection (Section \ref{subsec:relwork}).

\begin{figure}
	\centering
	\includegraphics[width=1\linewidth]{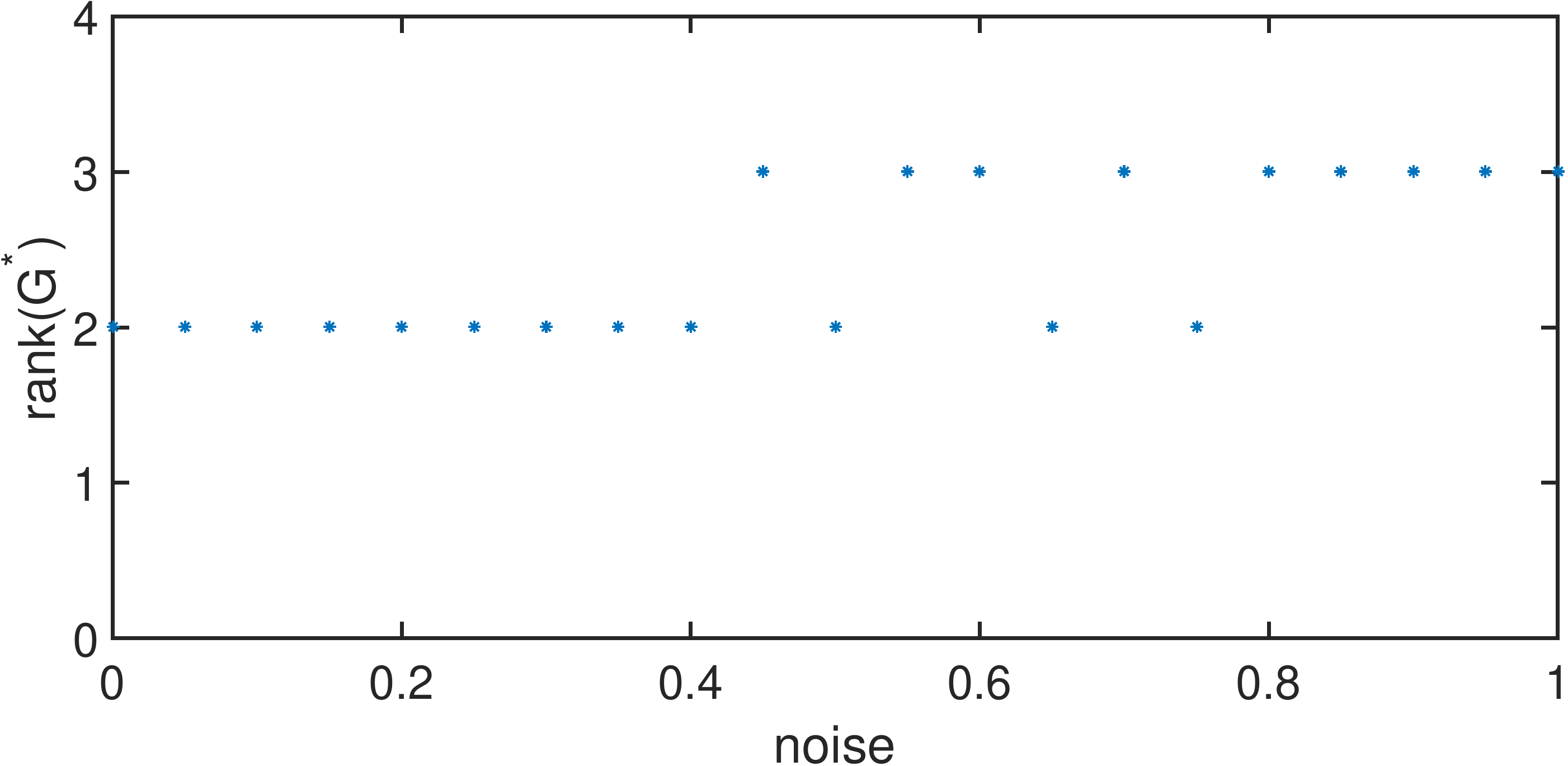}
	\caption
	{
		Phase transition for \ref{csdp}. The plot shows the rank of the global optimum $\bG^\ast$ as a function of the noise level. Below a certain noise threshold, the rank of $\bG^\ast$ is exactly $2$, i.e., the relaxation is tight. Above this  threshold, the rank of $\bG^\ast$ exceeds $2$, making it infeasible for \ref{sdp}. 
	}
	\label{fig:phasetran}
\end{figure}


\subsection{Related Work} \label{subsec:relwork}
The rank-restricted subset $\Omega$ of the PSD cone is nonconvex, which implies that standard convergence result for ADMM \cite{boyd} does not directly apply to \ref{ncadmm}. However, we do leverage the convergence of convex ADMM for analyzing the convergence of \ref{ncadmm} when the noise is low. A phase transition phenomena similar to the one cited above has been analyzed in \cite{bandeira}, albeit in the context of phase synchronization. Our proof of existence of the phase transition for  \ref{csdp}  is in the spirit of this analysis.

The theoretical convergence of ADMM for nonconvex problems has been studied in \cite{hong,chang,wotao}. However, a crucial working assumption common to these results does not hold in our case. More precisely, observe that we can rewrite \ref{sdp} as
\begin{equation} \label{resdp} 
\begin{aligned}
& \underset{\bG, \bH \in \S^{Md}}{\min} 
&& \mathrm{Tr}(\bC \bG) + \iota_{\Omega}(\bG) + \iota_{\Theta}(\bH) \\
&\text{subject to}               
&& \bG - \bH = \zer,
\end{aligned}
\end{equation}  
where $\iota_{\Gamma}$ is the indicator function associated with a feasible set $\Gamma$ \cite{boyd}, namely, $\iota_{\Gamma}(\bY) = 0$ if $\bY \in \Gamma$, and $\iota_{\Gamma}(\bY) = \infty$ otherwise. Notice that, because of the indicator functions, the objective function in \eqref{resdp} is non-differentiable in \emph{both} $\bG$ and $\bH$. This violates a regularity assumption common in existing analyses of nonconvex ADMM, namely, that the objective must be smooth in \emph{at least} one variable. In these works, convergence results are obtained by proving a monotonic decrease in the augmented Lagrangian. This requires: (i) bounding successive difference in dual variables by successive difference in primal variable, which is where the assumption of smoothness is used; (ii) requiring that the parameter $\rho$ is above a certain threshold. In particular, it is not clear whether this thresholding of the value of $\rho$ is fundamental to convergence, or just an artifact of the analysis. 

We do not make such smoothness assumptions in our analysis. We can afford to do this since we are analyzing a special class of problems, as opposed to the more general setups in \cite{hong,chang,wotao}. Instead of showing a monotonic decrease in the augmented Lagrangian, our analysis relies on the phenomenon of tightness of convex relaxation. This provides more insights into the convergence behavior of the algorithm. For instance, our explanation in Section \ref{subsec:osc} shows that the instability of the algorithm (in the high noise regime) for low values of $\rho$ is fundamental, while suggesting why this instability is not observed in the low noise regime.


\subsection{Organization}
The rest of the paper is organized as follows. The main results are discussed in Section \ref{sec:mainres}, which consists of six subsections. We conclude with a summary of our findings in Section \ref{sec:concl}. Appendix contains proofs of some subsidiary results that we use in our analysis.


\section{Convergence Analysis} \label{sec:mainres}
In this section we state and prove the main results of the paper. We set up notations and state some linear algebraic preliminaries in Section \ref{subsec:prelim}. We then prove some results on duality theory in Section \ref{subsec:dual} that we will use in subsequent analysis. In Section \ref{subsec:fixed}, we prove that any fixed point of \ref{ncadmm} is a KKT point of \ref{orthreg}. Sections \ref{subsec:tightness} and \ref{subsec:lonoise} deal with low noise regime: in Section \ref{subsec:tightness} we show the existence of, and explicitly compute, the noise threshold below which \ref{csdp} is tight, thus theoretically establishing the existence of the low noise regime; in Section \ref{subsec:lonoise} we analyze convergence of \ref{ncadmm} in this regime. Finally, in Section \ref{subsec:osc}, we give a duality-based explanation of the instability of \ref{ncadmm} in high noise regime for small values of $\rho$.

\subsection{Notations and Preliminaries} \label{subsec:prelim}
For a matrix $\bX$, $\norm{\cdot}$ and $\norm{\cdot}_2$ denote the Frobenius and spectral norms; the latter is simply the largest singular value $\sigma_{\mathrm{max}}(\bX)$. We note that $\norm{\bX}_2 \leq \norm{\bX}$. $\mathrm{Tr}(\bA)$ denotes the trace of $\bA$; and $\dotp{\bX,\bY} = \mathrm{Tr}(\bX \bY)$ is the inner product between two symmetric matrices $\bX$ and $\bY$. For a symmetric matrix $\bX$, $\lambda_i (\bX)$ denotes the $i$-th eigenvalue, where we assume the eigenvalues $(\lambda_i)$ to be arranged in a nondecreasing order.

Given a matrix $\bX$, we refer to the subspace spanned by its columns as the \emph{rangespace} of $\bX$; $\mathrm{rank}(\bX)$ is the dimension of the rangespace of $\bX$. The subspace spanned by the vectors $\v$ such that $\bX \v = \zer$ is referred to as the \emph{nullspace} of $\bX$; $\mathrm{nullity}(\bX)$ is the dimension of nullspace of $\bX$.

Following are some of the linear algebraic facts that we shall use in our analysis:

\begin{enumerate}
	\item Suppose $\bX, \bY \in \R^{n \times n}$, such that $\bX \bY = \zer$. If $\mathrm{rank}(\bY) = d$, then $\mathrm{nullity}(\bX) \geq d$.    Similarly, if $\mathrm{nullity}(\bX) = d$, then $\mathrm{rank}(\bY) \leq d$.
	
	\item Suppose $\bX, \bY \in \S^n_+$. Then, $\mathrm{Tr}(\bX \bY) = \zer \iff \bX \bY = \zer$. 
	
	\item (Weyl's Theorem \cite{bhatia}) Suppose $\bX, \bY \in \S^n$. Then for $i \in [1:n]$,
	\[
	\lambda_i(\bX) - \norm{\bY}_2 \leq \lambda_i(\bX + \bY) \leq \lambda_i(\bX) + \norm{\bY}_2.
	\]
	
	\item Suppose $\bX, \bY \in \S^n$, such that $\bX \bY = \zer$. Then $\bX$ and $\bY$ are simultaneously diagonalizable.
\end{enumerate}

We now note in the following proposition some properties of positive semidefinite matrices with identity matrix on the diagonal blocks. The proof is deferred to the Appendix.

\begin{proposition} \label{prop:Gprop}
	Suppose $\bG \in \S^{Md}_+$ and $[\bG]_{ii} = \bI_d$. Then
	\begin{enumerate} [$(a)$]
		\item $\mathrm{rank}(\bG) \geq d$;
		
		\item $\sigma_{\mathrm{max}} \left( [\bG]_{ij} \right) \leq 1, \:\: i,j \in [1:M]$;
		
		\item If $\mathrm{rank}(\bG) = d$, then any nonzero eigenvalue of $\bG$ is $M$.
	\end{enumerate}
\end{proposition}


\subsection{Duality} \label{subsec:dual}
We start by discussing some results on Lagrange duality, which plays an important role throughout our analysis. We use duality (i) for establishing KKT (Karush-Kuhn-Tucker) conditions at fixed points of \ref{ncadmm}, (ii) for establishing the existence of a low noise regime, (iii) for establishing convergence in low noise regime, and (iv) for explaining the phenomenon of oscillation observed at high noise. 
We first derive the KKT conditions for \ref{orthreg} by writing \ref{orthreg} as a nonlinear program \cite{bert}: 
\begin{equation} 
\label{regprob}
\begin{aligned}
& \underset{\bO_1,\cdots, \bO_M \in \R^{d\times d}}{\min} && \sum_{i,j=1}^M \mathrm{Tr} \left([\bC]_{ij} \bO_j^{\top} \bO_i \right)   \\
&\text{subject to}            
&& \bI_d-\bO_i^{\top} \bO_i = \zer, \qquad i \in [1:M].
\end{aligned}
\end{equation}
The Lagrangian for \eqref{regprob} is
\begin{equation} 
\label{eqn:lag}
\cL((\bO_i),(\Lam_i))= \sum_{i,j=1}^{M} \mathrm{Tr} \left([\bC]_{ij} \bO_j^{\top} \bO_i \right)  + \sum_{i=1}^{M} \mathrm{Tr}(\Lam_i\left( \mathbf{I}_d-\mathbf{O}_i^{\top} \mathbf{O}_i \!)\right),
\end{equation}
where the symmetric matrix $\Lam_i \in \mathbb{R}^{d \times d}$ is the multiplier for the $i$-th equality constraint in \eqref{regprob}.  The KKT point of \eqref{regprob} has the following characterization. (Recall that $\bG$ is the Gram matrix of the $\bO_i$'s, whose $(i,j)$-th block is $[\bG]_{ij} = \bO_i^{\top} \bO_j$.)

\begin{lemma}  \label{lem:kkt} 
	The variables $\bO^\ast_1,\ldots,\bO^\ast_M \in \R^{d\times d}$ constitute a KKT point of \eqref{regprob} if and only if, for $i \in [1:M]$,
	\begin{description}
		\item $(a)$ $[\bG^\ast]_{ii} = \mathbf{I}_d$, and
		\item $(b)$ $[\bC\bG^\ast]_{ii} = [\bG^\ast \bC]_{ii}$.
	\end{description}
	In this case, we say that $\bG^\ast$ is a KKT point of \textup{\ref{orthreg}}.
\end{lemma}
\begin{proof}
	For a minimization problem with equality constraints, KKT conditions amount to primal feasibility, and stationarity of Lagrangian \eqref{eqn:lag}  with respect to the primal variables \cite{bert} (i.e., the partial derivative of \eqref{eqn:lag}  with respect to $\bO_i$ should vanish). Primal feasibility gives us condition (a). Vanishing of partial derivative of \eqref{eqn:lag} with respect to $\bO_i$ gives us
	\begin{equation}
	\label{temp}
	\bO_i^\ast \Lam_i^\ast = -\sum_{j=1}^{M} \bO_j^\ast [\bC]_{ji}, \quad i \in [1:M].
	\end{equation}
	Left multiplying by $\bO_i^{\ast\top}$ and using the primal feasibility condition that $\bO_i^{\ast\top} \bO_i^\ast = \bI_d$, we obtain
	\[
	\Lam_i^\ast = -\sum_{j=1}^{M} \bO_i^{\ast\top} \bO_j^\ast [\bC]_{ji}  =  -\sum_{j=1}^{M} [\bG^\ast]_{ij} [\bC]_{ji}  = -[\bG^\ast \bC]_{ii}.
	\]
	Also, note that $\Lam_i^{\ast\top} = -[\bC \bG^\ast]_{ii}$. Since $\Lam_i^\ast$ is symmetric, condition (b) follows immediately. 
	Conversely, given conditions (a) and (b) on $\bG^\ast$, it is not difficult to see that $\Lam_i$'s defined as $\Lam_i = -[\bC \bG^\ast]_{ii}$, and $\bO_i$'s defined such that $\bO_j^\top \bO_k = [\bG^\ast]_{jk} \ \forall j,k \in [1:M]$, together satisfy the KKT conditions.
	\qed
\end{proof}

It is not difficult to check that the gradients of the constraints in \eqref{regprob} are linearly independent. As a result, the KKT conditions necessarily hold at any local minimum of \eqref{regprob} \cite{bert}.
We now give the  KKT conditions for the convex program \ref{csdp}.

\begin{lemma}  \label{lem:kktrel} 
		$\bG^\ast$ is a KKT point of \textup{\ref{csdp}} if there exists a block diagonal matrix $\Lam^\ast \in \S^{Md}$ such that
	\begin{description}
		\item $(a)$ $\bG^\ast \in \S^{Md}_+$, $[\bG^\ast]_{ii} = \bI_d$,
		\item $(b)$ $\bC+\Lam^\ast \in \mathbb{S}_{+}^{Md}$, and
		\item $(c)$ $(\bC+\Lam^\ast)\:\bG^\ast = \zer$.
	\end{description}
	(We call $\Lam^\ast$ the dual variable corresponding to $\bG^\ast$.)
\end{lemma} 
\begin{proof}
	KKT conditions for the convex program \ref{csdp} are just the conditions of primal feasibility, dual feasibility, and complementary slackness \cite{boydbook}.
	Condition (a) is the primal feasibility condition, condition (b) is the dual feasibility condition, and condition (c) is the complementary slackness condition that follows from strong duality \cite{aliza}. Now, to see why $\Lam^\ast$ is block diagonal, we write \ref{csdp} as a standard semidefinite program \cite{aliza}
	\[
	\begin{aligned}
	&\underset{\bG}{\min} \ && \mathrm{Tr} (\bC \bG) \\
	& \mathrm{s.t.} 		&& \mathrm{Tr} (\bA_k \bG) = \b_k, k = [1:m], \\
	&						&& \bG \in \S^{Md}_+.
	\end{aligned}
	\]
	Here $\bA_k \in \S^{Md}$, and $\mathrm{Tr} (\bA_k \bG) = \b_k, k = [1:m]$ collectively encode the condition that $[\bG]_{ii} = \bI_d$. From \cite{aliza}, $\Lam^\ast$ is of the form $\sum_{k=1}^{m} y_k \bA_k$, for some scalars $y_k$, and it is not difficult to see that $\sum_{k=1}^{m} y_k \bA_k$ forms a symmetric block diagonal matrix in this case.
	\qed
\end{proof}

Observe that the identity matrix $\bI_{Md}$ is strictly feasible for \ref{csdp}. Thus, by Slater's condition \cite{boydbook}, the KKT conditions in Lemma \ref{lem:kktrel} are both necessary and sufficient for $\bG^\ast$ to be a global minimizer of \ref{csdp} \cite{boydbook}. Note that we refer to $\Lam^\ast$ as \emph{the} dual variable corresponding to $\bG^\ast$: this is because, given a global minimizer $\bG^\ast$, there is a unique block diagonal matrix $\Lam^\ast$ satisying conditions in Lemma \ref{lem:kktrel}. This is formalized in the following proposition, which expresses dual optimum in terms of the primal optimum. The proof is deferred to the Appendix.

\begin{proposition} \label{prop:dualform}
	Let $\bG^\ast$ and $\Lam^\ast$ be as in Lemma \ref{lem:kktrel}. Then
	\begin{equation}
	[\Lam^\ast]_{ii} = -[\bC \bG^\ast]_{ii}, \: i \in [1:M].
	\end{equation}
\end{proposition}

Finally, we consider the dual variables generated by \ref{cadmm} and \ref{ncadmm}. For both these algorithms, we will assume that the dual initialization $\Lam^0$ is block diagonal.

\begin{proposition} \label{prop:lamblockdiag}
	Let $\Lam^0$ be a block diagonal matrix in \textup{\ref{ncadmm}} (or \textup{\ref{cadmm}}). Then
	\begin{enumerate} [(i)]
		\item $\Lam^k$ is block diagonal at every $k \geq 1$;
		\item The $\bH$-update step in \eqref{ncadmmproj} reduces to 
		\begin{equation} \label{hupdate1}
		\bH^{k+1} = \Pi_\Theta \left( \bG^{k+1} \right).
		\end{equation}
	\end{enumerate}
\end{proposition}
The proof is deferred to the Appendix.


\subsection{Fixed Point Analysis} \label{subsec:fixed}
The first result we prove is that any fixed point of \ref{ncadmm} is a stationary (KKT) point of \ref{orthreg}. Before we do so, we clarify what we mean by a ``fixed point''. 

\begin{definition}
	Suppose $\cT$ is an iterative procedure initialized with $\z^0$, and working as follows: $\z^{k+1} = \cT(\z^k)$. Then, $\z^{k_0}$ is a \emph{fixed point} of $\cT$ if $\z^{k_0} = \cT(\z^{k_0})$.
\end{definition}

In other words, the iterative procedure stops updating the variables when the fixed point is reached. We now characterize fixed point of \ref{ncadmm}.

\begin{proposition} \label{prop:fixedpt}
	Suppose \textup{\ref{ncadmm}} is initialized with $(\bH^0, \Lam^0)$, where $\Lam^0$ is a block diagonal matrix. Then $(\bH^{k_0}, \Lam^{k_0})$ is a fixed point of \textup{\ref{ncadmm}} if and only if $\bG^{k_0+1} = \bH^{k_0}$.
\end{proposition}

\begin{proof}
	If $\bG^{k_0+1} = \bH^{k_0}$, then clearly $[\bG^{k_0+1}]_{ii} = \bI_d$. This means that $\bH^{k_0+1} = \Pi_\Theta \left( \bG^{k_0+1} \right) = \bG^{k_0+1} = \bH^{k_0}$. This, in turn, implies that $\Lam^{k_0+1} = \Lam^{k_0} + \rho (\bG^{k_0+1} - \bH^{k_0+1}) = \Lam^{k_0}$. Thus, we have that  $(\bH^{k_0+1}, \Lam^{k_0+1}) = (\bH^{k_0}, \Lam^{k_0})$.
	
	Conversely, if $(\bH^{k_0}, \Lam^{k_0})$ is a fixed point of \ref{ncadmm}, then by definition $(\bH^{k_0+1}, \Lam^{k_0+1}) = (\bH^{k_0}, \Lam^{k_0})$. Now, $\Lam^{k_0+1} = \Lam^{k_0}$ implies that $\bG^{k_0+1} - \bH^{k_0+1} = \zer$. Thus, we have that $\bG^{k_0+1} = \bH^{k_0+1} = \bH^{k_0}$.
	\qed
\end{proof}

We now prove the following theorem, which, combined with Proposition \ref{prop:fixedpt}, implies stationarity of any fixed point of \ref{ncadmm}.

\begin{theorem} \label{thm:fixpt}
Suppose \textup{\ref{ncadmm}} is initialized with $(\bH^0, \Lam^0)$, where $\Lam^0$ is a block diagonal matrix. Suppose $\bG^{k_0+1}=\bH^{k_0}$. Then $\bH^{k_0}$ is a KKT point of \textup{\ref{orthreg}}.
\end{theorem}

\begin{proof}
	Recall that $\bG^{k_0+1}$ is formed from top-$d$ eigendecomposition of $\bH^{k_0} - \rho^{-1} \left(\bC+\Lam^{k_0}\right)$. This implies that the rangespace of $\left( \bH^{k_0} - \rho^{-1} \left(\bC+\Lam^{k_0}\right) - \bG^{k_0+1} \right)$ is orthogonal to rangespace of $\bG^{k_0+1}$, which gives us
	\begin{equation} \label{optimality}
	\left( \bH^{k_0} - \rho^{-1} \left(\bC+\Lam^{k_0}\right) - \bG^{k_0+1} \right) \bG^{k_0+1} = \zer.
	\end{equation}
	Or,	
	\begin{equation} \label{fixedpt}
	\left( \bH^{k_0} - \bG^{k_0+1} \right) \bG^{k_0+1} =  \rho^{-1} \left(\bC+\Lam^{k_0}\right) \bG^{k_0+1}.
	\end{equation}
	Since $\bG^{k_0+1}=\bH^{k_0}$, the left hand side of \eqref{fixedpt} is $\zer$, and we have
	\begin{equation} \label{eqn:kkt}
	\begin{aligned}
	&\left(\bC+\Lam^{k_0}\right) \bH^{k_0} = \zer. \\
	\Rightarrow \quad & \bC \bH^{k_0} \rightarrow -\Lam^{k_0} \bH^{k_0}. \\
	\Rightarrow \quad & [\bC \bH^{k_0}]_{ii} \rightarrow -[\Lam^{k_0} \bH^{k_0}]_{ii}.
	\end{aligned}
	\end{equation}
	Observe that $[\Lam^{k_0} \bH^{k_0}]_{ii}$ is symmetric (since $\Lam^{k_0}$ is block diagonal, and $[\bH^{k_0}]_{ii} = \bI_d$). Thus, we have that $[\bC \bH^{k_0}]_{ii}$ is symmetric. Using Lemma \ref{lem:kkt}, we can now conclude that $\bH^{k_0}$ is a KKT point of \ref{orthreg}.
	\qed
\end{proof}


\subsection{Existence of Low Noise Regime: Tightness of \ref{csdp}} \label{subsec:tightness}
In Section \ref{subsec:lonoise}, we will derive convergence results for \ref{ncadmm} in the low noise regime. Recall from our discussion in Section \ref{subsec:contri}, that we say we are in the low noise regime when the convex relaxation \ref{csdp} is tight; that is, when global minimizer of \ref{csdp} is also a global minimizer of \ref{sdp}. Before we prove convergence results for \ref{ncadmm} in the low noise regime, it is imperative that we theoretically establish the empirical observation (Fig. \ref{fig:phasetran}) that there is indeed a noise threshold below which \ref{csdp} is tight. This is the task for this section. 
Recall that the data matrix $\mathbf{C}$ in \ref{sdp} ultimately depends on the local coordinate measurements $\x_{k,i}$, $i \in [1:M]$, $k \in P_i$. If these measurements are exact, we say that $\mathbf{C}$ is clean; otherwise, we say that $\mathbf{C}$ is noisy. 

Our analysis for tightness of \ref{sdp} relies on the assumption that nullity of the clean data matrix is $d$. Before stating the main results of this section, we examine some properties of the clean data matrix that will help demystify this assumption.
To do so, we rely on our discussion in Section \ref{subsec:probstat}. Let $\bar{\x}_1,\ldots,\bar{\x}_N \in \R^d$ be the true coordinates of the $N$ nodes in some global coordinate system. Let the local coordinate system associated with patch $P_i$ be related to the global coordinate system via rigid transform $(\bar{\bO}_i,\bar{\t}_i)$, $i \in [1:M]$. It follows from equation \eqref{exact} that
\begin{equation} 
\sum_{i=1}^{M} \sum_{k \in P_i} \norm{\bar{\x}_k -  \left( \bar{\bO}_i \x_{k,i} + \bar{\t}_i \right) }^2 = 0.
\end{equation}
Let $\bC_0$ be the clean data matrix formed from exact local coordinate measurements. Define $\bG_0$ be defined by 
\begin{equation} \label{g0}
[\bG_0]_{ij} = \bar{\bO}_i^\top \bar{\bO}_j, \quad i,j \in [1:M]. 
\end{equation}
Then, from \eqref{equiv}, and from the fact that \ref{sdp} is an exact reformulation of \ref{orthreg}, we get that
\begin{equation} \label{trcg}
\mathrm{Tr}\left( \bC_0 \bG_0 \right) = 0.
\end{equation}
Now, since $\bC_0, \bG_0 \in \S^{Md}_+$, we get that
\begin{equation} \label{cg}
\bC_0 \bG_0 = \zer.
\end{equation}

From \eqref{cg} and the fact that $\mathrm{rank}(\bG_0) = d$, it follows that $\mathrm{nullity}(\bC_0) \geq d$. The assumption we make on $\bC_0$ is that $\mathrm{nullity}(\bC_0) = d$. This is equivalent to the condition that the \emph{body graph} corresponding to the registration problem is \emph{affinely rigid} \cite{knc}. Loosely speaking, the registration problem is well-posed if $\mathrm{nullity}(\bC_0) = d$. For example, this is true when each patch contains all the $N$ points (the setup for the simulations in Section \ref{subsec:numexp}). More generally, $\mathrm{nullity}(\bC_0) = d$ if there exists an ordering of the patches such that $P_1$ contains at least $d+1$ points, and $P_i$ has at least $d+1$ points in common with $P_1 \cup P_2 \cup \cdots \cup P_{i-1}$, for $i \geq 2$ \cite{knc} (this occurs naturally in applications like multiview registration \cite{miraj}).

We now state a result that we require to prove the main result of this section (Theorem \ref{thm:tight}). This result is also of independent interest in the context of stability of \ref{csdp}. In particular, the lemma quantifies the perturbation in the optimum of \ref{csdp} due to a perturbation in $\bC_0$. We defer the proof to the Appendix.

\begin{lemma} \label{lem:stable}
	Suppose $\mathrm{nullity}(\bC_0) = d$. Let $\bC$ be the corresponding noisy data matrix, and let $\bW = \bC - \bC_0$. Let $\bG^\ast$ be a global optimum for \textup{\ref{csdp}} corresponding to the noisy data $\bC$. Let $\Delta = \bG^\ast - \bG_0$, where $\bG_0$ is the ground-truth solution. Then
	\[
	\norm{\Del} \leq \frac{2M}{\lambda_{d+1}(\bC_0)} \norm{\bW}.
	\]
\end{lemma}

We now state the main result of this section concerning tightness of \ref{csdp}.

\begin{theorem} \label{thm:tight}
	Suppose $\mathrm{nullity}(\bC_0) = d$. Let $\bC = \bC_0 + \bW$ be the noisy data matrix. Then there exists (explicitly computable) $\eta > 0$ such that $\mathrm{rank}(\bG^\ast) = d$ (i.e. \textup{\ref{csdp}} is tight) if $\norm{\bW} < \eta$.
\end{theorem}
\begin{proof}
	Let $\Lam^\ast$ be the dual corresponding to $\bG^\ast$. 
	From the fact that $(\bC + \Lam^\ast) \bG^\ast = \zer$ (Lemma \ref{lem:kktrel}(b)), and that $\mathrm{rank}(\bG^\ast) \geq d$ (Proposition \ref{prop:Gprop}(a)), we have that $\mathrm{nullity}(\bC + \Lam^\ast) \geq d$. 
	This also tells us the following: if we prove $\mathrm{nullity}(\bC + \Lam^\ast) = d$, then it necessarily follows that $\mathrm{rank}(\bG^\ast) = d$, implying that \ref{csdp} is tight.
	Thus, our goal is to prove that there exists $\eta > 0$ such that $\mathrm{nullity}(\bC + \Lam^\ast) = d$ if $\norm{\bW} < \eta$.
	
	From Proposition \ref{prop:dualform}, we have that
	\[	
	\begin{aligned}
	\bC + \Lam^\ast 
	&= \bC - \mathrm{bd}(\bC \bG^\ast) \\
	&= \bC_0 + \bW - \mathrm{bd}((\bC_0 + \bW) \bG^\ast) \\
	&= \bC_0 + \bW - \mathrm{bd}(\bC_0 \bG^\ast) - \mathrm{bd}(\bW \bG^\ast) \\
	&= \bC_0 + \underbrace{\bW - \mathrm{bd}(\bC_0 (\bG^\ast - \bG_0)) - \mathrm{bd}(\bW \bG^\ast)}_\text{$\bA$}\\
	\end{aligned}
	\]
	Here $\mathrm{bd}: \R^{Md \times Md} \rightarrow \R^{Md \times Md}$ is the linear operator that leaves the diagonal blocks untouched, and sets every other element to $0$. To get the last equality, we have added a superfluous term $\mathrm{bd}(\bC_0 \bG_0)$. We can do this because $\bC_0 \bG_0 = \zer$, see \eqref{cg}. 
	
	Thus, in summary, $\bC + \Lam^\ast = \bC_0 + \bA$, where $\bA = \bW - \mathrm{bd}(\bW \bG^\ast) - \mathrm{bd}(\bC_0 \Del)$, and $\Del = \bG^\ast - \bG_0$ is the perturbation in the optimum of \ref{csdp} due to pertubation in the data matrix. Now, by Weyl's theorem,
	\begin{equation} \label{weyl}
	\lambda_{d+1}(\bC + \Lam^\ast) \geq \lambda_{d+1}(\bC_0) - \norm{\bA}_2.
	\end{equation}
	We have already established that $\mathrm{nullity}(\bC + \Lam^\ast) \geq d$. This implies that the bottom $d$ eigenvalues of $(\bC + \Lam^\ast)$ are guaranteed to be $0$, i.e., $\lambda_{1}(\bC + \Lam^\ast) = \cdots = \lambda_{d}(\bC + \Lam^\ast) = 0$. To prove $\mathrm{nullity}(\bC + \Lam^\ast) = d$, we need to show that $\lambda_{d+1}(\bC + \Lam^\ast) > 0$. We can do this using \eqref{weyl}, if we upperbound $\norm{\bA}_2$.
	\[
	\begin{aligned}
	\norm{\bA}_2 
	&= \norm{\bW - \mathrm{bd}(\bW \bG^\ast) - \mathrm{bd}(\bC_0 \Del)}_2 \\
	&\leq \norm{\bW} + \norm{\mathrm{bd}(\bW \bG^\ast)} + \norm{\mathrm{bd}(\bC_0 \Del)} \\
	&\leq \norm{\bW} + \norm{\bW \bG^\ast} + \norm{\bC_0 \Del} \\
	&\leq \norm{\bW} \left( 1 + \norm{\bG^\ast} \right) + \norm{\bC_0}\,\norm{\Del},
	\end{aligned}
	\]
	where we have used the results that $\norm{\cdot}_2 \leq \norm{\cdot}$, that $\norm{\cdot}$ obeys triangle inequality, that $\norm{\mathrm{bd}(\cdot)} \leq \norm{\cdot}$, and that $\norm{\cdot}$ is sub-multiplicative \cite{bhatia}. We now need an upper bound on $\norm{\Del}$, which is the magnitude of perturbation in the optimum of \ref{csdp} when the data matrix is perturbed from $\bC_0$ to $\bC = \bC_0 + \bW$. This is where we use Lemma \ref{lem:stable}, which gives us
	\[
	\norm{\Del} \leq \frac{2M}{\lambda_{d+1}(\bC_0)} \norm{\bW}.
	\]
	Thus, we have
	\[
	\begin{aligned}
	\norm{\bA}_2 
	&\leq \norm{\bW} \, \left(1 + \norm{\bG^\ast} + \frac{2M}{\lambda_{d+1}(\bC_0)} \norm{\bC_0}\right) \\
	&\leq \norm{\bW} \, \left(1 + M \sqrt{d} + \frac{2M}{\lambda_{d+1}(\bC_0)} \norm{\bC_0}\right),
	\end{aligned}
	\]		
	where, the fact that $\norm{\bG^\ast} \leq M \sqrt{d}$ follows from Proposition \ref{prop:Gprop}(b). The terms in the paranthesis are constant. Thus, by controlling $\norm{\bW}$, we can control $\norm{\bA}_2$. In particular, if
	\begin{equation} \label{eta}
	\norm{\bW} < \eta = \frac{\lambda_{d+1}(\bC_0)}{\left(1 + M \sqrt{d} + \frac{2M}{\lambda_{d+1}(\bC_0)} \norm{\bC_0}\right)},
	\end{equation}	
	then $\lambda_{d+1}(\bC_0) - \norm{\bA}_2 > 0$. This, from equation \eqref{weyl}, implies that $\lambda_{d+1}(\bC + \Lam^\ast) > 0$.
	\qed
\end{proof}	 

When the noise is low enough for the convex relaxation \ref{csdp} to be tight, we say that we are in the \emph{low noise regime}. Theorem \ref{thm:tight} establishes the validity of this notion; that is, we do have a non-zero noise level $\eta$, given by \eqref{eta}, below which \ref{csdp} is tight.


\subsection{Convergence in Low Noise Regime} \label{subsec:lonoise}

We start with a result on \ref{cadmm}, which we will leverage to get results for \ref{ncadmm} in the low noise regime.
Suppose $\bA^k = \bH^k - \rho^{-1}(\bC + \Lam^k)$, $\bA^* = \bH^* - \rho^{-1}(\bC + \Lam^*)$, where $\bH^\ast$ is a global minimum for the convex relaxation \ref{csdp}, and $\Lam^\ast$ is the corresponding optimal dual as in Lemma \ref{lem:kktrel}.

\begin{lemma} \label{lem:cadmm}
	For \textup{\ref{cadmm}}, $\norm{\bA^{k+1} - \bA^*}^2 \leq \norm{\bA^{k} - \bA^*}^2$, for every $k \geq 0$.
\end{lemma}

Proof of Lemma \ref{lem:cadmm} is deferred to the Appendix. 
We now state our result on convergence of \ref{ncadmm} in the low noise regime ($\bW$ and $\eta$ are as in Theorem \ref{thm:tight}).

\begin{theorem} \label{thm:lownoise}
	Suppose $\norm{\bW} < \eta$, which implies that the convex relaxation is tight, i.e., $\mathrm{rank}(\bG^\ast) = d$, where $\bG^\ast$ is a global optimum for \textup{\ref{csdp}}. Let $\Lam^\ast$ be the corresponding optimum dual variable as in Lemma \ref{lem:kktrel}. Suppose $\bH^0$ and $\Lam^0$ are such that
	\begin{equation} \label{lonoiseinit}
	\norm{\bH^0 - \bG^\ast}^2 + \frac{1}{\rho^2} \norm{\Lam^0 - \Lam^\ast}^2 \leq \frac{1}{\rho^2} \, \lambda^2_{d+1}(\bS^\ast).
	\end{equation}
	Then \textup{\ref{ncadmm}}  converges to the global optimum, that is, $\bH^k \rightarrow \bG^\ast$ and $\Lam^k \rightarrow \Lam^\ast$ as $k \rightarrow \infty$.
\end{theorem}

\begin{proof}
	Since $\norm{\bW} < \eta$, we have that $\mathrm{nullity}(\bC + \Lam^\ast) = d$ (see proof of Theorem \ref{thm:tight}), which, in particular implies that $\lambda_{d+1}(\bC + \Lam^\ast) > 0$. We now deduce the eigenvalues of $\bA^\ast$. From KKT condition (a) in Lemma \ref{lem:kktrel}, we have that 
	\[
	(\bC + \Lam^\ast)\: \bG^\ast = \zer.
	\]
	In particular, this means that $\bC + \Lam^\ast$, and $\bG^\ast$ are simultaneously diagonalizable, and since their product is $\zer$, their rangespaces are orthogonal. Now, since $\mathrm{rank}(\bG^\ast) = d$, the non-zero eigenvalues of $\bG^\ast$ are all $M$ (Proposition \ref{prop:Gprop}(c)). Moreover, from KKT condition (b) in Lemma \ref{lem:kktrel}, we have that $\bC + \Lam^\ast \in \S^{Md}_+$. Putting everything together, we have that
	\begin{itemize}
		\item $\bA^\ast$ has $d$ positive eigenvalues, each of them equal to $M$.
		\item $\bA^\ast$ has $(M-1)d$ negative eigenvalues, which are the negative of the non-zero eigenvalues of $\rho^{-1} (\bC + \Lam^\ast)$.
	\end{itemize}
	Let
	\[
	\cB^\ast := \{ \bA \in \S^{Md} : \norm{\bA - \bA^\ast} < \rho^{-1} \lambda_{d+1}(\bC + \Lam^\ast) \}.
	\]	
	By Weyl's theorem, we deduce that for any $\bA \in \cB^\ast$, the top $d$ eigenvalues of $\bA$ would lie in the interval $\left( M -\rho^{-1} \lambda_{d+1}(\bC + \Lam^\ast), M + \rho^{-1} \lambda_{d+1}(\bC + \Lam^\ast) \right)$, and other eigenvalues of $\bA$ would lie in the interval $\big( -\rho^{-1} \left( \lambda_{Md}(\bC + \Lam^\ast) + \lambda_{d+1}(\bC + \Lam^\ast) \right), 0 \big)$. In other words, for any $\bA \in \cB^\ast$, only the top $d$ eigenvalues of $\bA$ would be nonnegative.
	Let
	\[
	\bA^k = \bH^k - \rho^{-1} \left( \bC + \Lam^k \right).
	\]
	Suppose $\bA^k \in \cB^\ast$. Then, because only the top $d$ eigenvalues of $\bA^k$ are nonnegative, we have
	\[
	\bG^{k+1} = \Pi_\Omega \left( \bA^k \right) = \Pi_{\S^{Md}_+} \left( \bA^k \right).
	\]
	That is, projection on the nonconvex set $\Omega$ is same as projection on the convex set $\S^{Md}_+$. 
	
	Now, from Lemma \ref{lem:cadmm} we infer that, if $\bA^k \in \cB^\ast$, then $\bA^{k+1} \in \cB^\ast$. Thus, every subsequent projection on $\Omega$ is equivalent to projection on $\S^{Md}_+$. Note that condition \eqref{lonoiseinit} in the hypothesis of the theorem implies that $\bA^0 \in \cB^\ast$. Thus, the iterates generated by \ref{ncadmm} initialized with $\bH_0$, $\Lam_0$ is the same as the iterates generated by \ref{cadmm} initialized with $\bH_0$, $\Lam_0$. Now, since \ref{cadmm} converges to global optimum \cite{kncspl}, we deduce that \ref{ncadmm} converges to global optimum.
	\qed
\end{proof}

In summary, when the conditions in Theorem \ref{thm:lownoise} are satisfied, \ref{ncadmm} and \ref{cadmm} generate the same iterates, which allows us to infer the convergence of \ref{ncadmm} from that  of \ref{cadmm}. Thus, we reap the convergence benefits of \ref{cadmm}, while retaining the computational advantages of \ref{ncadmm} (see discussion in Section \ref{subsec:cadmm}).

Observe that Theorem \ref{thm:lownoise} implies a tradeoff, namely, if we intialize the dual sufficiently close to the optimal, we can be lax with the primal initialization. This principle is brought to the fore in the clean case, where we know that the dual optimum $\Lam^\ast = \zer$ (this is because, in the clean case, $\bG_0$ is the primal global optimum, and $\bC_0 \bG_0 = \zer$; see \eqref{cg} and Proposition \ref{prop:dualform}).

\begin{theorem} \label{thm:cleancase}
	Let $\mathrm{nullity}(\bC_0) = d$ and $\Lam^0 = \zer$. Then given any primal initialization $\bH^0$, \textup{\ref{ncadmm}} converges to global optimum provided
	\[
	\rho \leq \frac{\lambda_{d+1}(\bC_0)}{\sqrt{\norm{\bH^0}^2 + 2M \sqrt{d} \, \norm{\bH^0} + M^2 d}}.
	\]
\end{theorem}
\begin{proof}
	For the clean case, $\Lam^\ast = \zer$. Thus, with $\Lam^0 = \zer$, the condition in Theorem \ref{thm:lownoise} reduces to 
	\[
	\norm{\bH^0 - \bG^\ast}^2 \leq \frac{1}{\rho^2} \, \lambda^2_{d+1}(\bC_0).
	\]
	Now,
	\[
	\norm{\bH^0 - \bG^\ast}^2 \leq \norm{\bH^0}^2 + \norm{\bG^\ast}^2 + 2\norm{\bH^0} \, \norm{\bG^\ast}.
	\]
	Since $\norm{\bG^\ast} = M \sqrt{d}$, we have
	\[
	\norm{\bH^0 - \bG^\ast}^2 \leq \norm{\bH^0}^2 + 2M \sqrt{d} \, \norm{\bH^0} + M^2 d.
	\]
	Thus, the algorithm would converge to the global optimum if
	\[
	\norm{\bH^0}^2 + 2M \sqrt{d} \, \norm{\bH^0} + M^2 d \leq \frac{1}{\rho^2} \, \lambda^2_{d+1}(\bC_0),
	\]
	which proves the result.
	\qed
\end{proof}

In Fig.  \ref{fig:cleanrho}, \ref{ncadmm} gets stuck in a local optimum for $\rho = 1$, but converges to the global optimum if  $\rho$ is set using Theorem \ref{thm:cleancase}. The result also sheds light on the robustness of the algorithm to primal initializations, as observed in \cite{miraj}. Note that we have an upper bound on $\rho$, in contrast to existing results in the literature which prescribe a lower bound. This phenomenon is peculiar to low noise data. In contrast, small values of $\rho$ in the high noise regime may lead to non-attenuating oscillations (Fig. \ref{fig:hinoiserho}).


\subsection{Oscillations in High Noise Regime} \label{subsec:osc}
Before discussing oscillations in the high-noise regime, we will derive a necessary condition that a fixed-point of \ref{ncadmm} must satisfy.

\begin{proposition} \label{prop:fixedpt1}
	Suppose $(\bH^\ast,\Lam^\ast)$ is a fixed-point of \textup{\ref{ncadmm}} algorithm with parameter $\rho$. Then $(\bC + \Lam^\ast) \bH^\ast = \zer$.
\end{proposition}
\begin{proof}
	Suppose we initialize \ref{ncadmm} with $\bH^0 = \bH^\ast$, $\Lam^0 = \Lam^\ast$. Then, by Proposition \ref{prop:fixedpt}, $\bG^1 = \bH^0$. Now, since $\bG^1$ is formed from top-$d$ eigendecomposition of $\bH^0 - \rho^{-1} \left(\bC+\Lam^0\right)$ (see \eqref{ncadmmproj}), the range space of $\left( \bH^0 - \rho^{-1} \left(\bC+\Lam^0\right) - \bG^1 \right)$ is orthogonal to range space of $\bG^1$, which gives us
	\begin{equation} 
	\left( \bH^0 - \rho^{-1} \left(\bC+\Lam^0\right) - \bG^1 \right) \bG^1 = \zer.
	\end{equation}
	Or,	
	\[
	\rho^{-1} \left(\bC+\Lam^0\right) \bG^1 = \left( \bH^0 - \bG^1 \right) \bG^1.
	\]
	Since $\bG^1 = \bH^0 \ (= \bH^\ast)$, the right hand side of the preceding equation is $\zer$, and we obtain the desired result.
	\qed
\end{proof}

Let us now focus on the high-noise regime. Suppose \ref{ncadmm} is initialized with $(\bH^0,\Lam^0)$ such that $\bH^0$ is feasible for \ref{sdp}, $\Lam^0$ is block-diagonal, and $(\bC + \Lam^0) \bH^0 = \zer$.
In the high-noise regime (i.e., when the convex relaxation is not tight), $\bC + \Lam^0$ must have at least one negative eigenvalue (say $-\mu^2$). Indeed, if $\bC + \Lam^0$ were positive semidefinite, it would follow from Lemma \ref{lem:kktrel} that the convex relaxation is tight. Now, the $\bG$-update is
\[
\bG^{1} = \Pi_\Omega \left( \bH^0 - \rho^{-1} (\bC + \Lam^0) \right).
\]
From Proposition \ref{prop:Gprop}, we know that the $d$ non-zero eigenvalues of $\bH^0$ are $M$. Now, observe that $-\rho^{-1} (\bC + \Lam^0)$ has a positive eigenvalue $\rho^{-1}\mu^2$. If $\rho$ is sufficiently small, we would have $\rho^{-1}\mu^2 > M$. Then, since $\bG^{1}$ is determined by the top $d$ eigenvalues of $\bH^0 - \rho^{-1} (\bC + \Lam^0)$, we would have that the top eigenvalue of $\bG^{1}$ is strictly bigger than $M$, making $\bG^{1}$ infeasible for \ref{sdp} (since, by Proposition \ref{prop:Gprop}, any $\bG$ feasible for \ref{sdp} necessarily has all non-zero eigenvalues equal to $M$). This would imply that $\bH^1 \neq \bG^1$, and consequently, that $\Lam^1 \neq \Lam^0$. That is, for small value of $\rho$, we see that \ref{ncadmm} does not stabilize even when $\bH^0$ and $\Lam^0$ satisfy $(\bC + \Lam^0) \bH^0 = \zer$, a property that any candidate for a fixed-point of \ref{ncadmm} must satisfy. Put differently, there can be no fixed-point to which \ref{ncadmm} converges if $\rho$ is sufficiently small.

Observe that the argument above depended on the existence of a negative eigenvalue of $\bC + \Lam^0$. This argument does not hold in the low-noise regime because we can simultaneously have the properties that $(\bC + \Lam^0) \bH^0 = \zer$, and that all the eigenvalues of $\bC + \Lam^0$ are nonnegative (which holds when $\bH^0$ is global optimum; see condition (b) in Lemma \ref{lem:kktrel}). This suggests why the instability is not observed for low values of $\rho$ in the low-noise regime.

\section{Conclusion} \label{sec:concl}
In this work, we analyzed the convergence behavior of \ref{ncadmm} in different noise regimes. Existing results on nonconvex ADMM do not apply to \ref{ncadmm} as they rely on certain smoothness assumptions that are not satisfied by \ref{sdp}. We bypassed these assumptions, and exploited the tightness phenomenon of convex relaxation to guide our analysis. We started with a fixed point analysis, where we proved that any fixed point of \ref{ncadmm} converges to a stationary (KKT) point of \ref{sdp}. To further refine the result, we looked at the behavior of \ref{ncadmm} when the noise is low. In particular, we precisely defined  what is meant by ``low'' noise by establishing tightness of the convex relaxation \ref{csdp} below a certain noise threshold. We then proved that, by initializing the primal and dual variables sufficiently close to the optimum, the iterates of \ref{ncadmm} are guaranteed to converge to the global optimum. By applying this result to the clean case, we showed that given \emph{any} primal initialization, we can explicitly compute values of $\rho$ for which the algorithm converges to the global optimum. For high noise, we showed that for sufficiently small $\rho$, the iterates generated by \ref{ncadmm} do not stabilize, even if initialization of \ref{ncadmm} satisfies necessary property of a fixed-point.


\section{Appendix}

\subsection{Proof of Proposition \ref{prop:Gprop}}
	$(a)$ 
	Consider the first $d$ columns of $\bG$. Since $\bG_{11} = \bI_d$, we conclude that the first $d$ columns of $\bG$ are linearly independent.
	
	$(b)$	
	Let $\x_i$, $\x_j \in \R^d$, $\norm{\x_i} = \norm{\x_j} = 1$ . Regard any vector in $\R^{Md}$ as consisting of $M$ vectors in $\R^d$ stacked vertically. Construct $\x \in \R^{Md}$ with $i$-th block as $\x_i$, $j$-th block as $-\x_j$, and every other block as $\zer$. Then
	\[
	\x^\top \bG \x = \norm{\x_i}^2 + \norm{\x_j}^2 - 2 \x_i^\top [\bG]_{ij} \x_j,
	\]
	where we have used the fact that $[\bG]_{ji} = [\bG]_{ij}^\top$. Since $\bG \in \S^{Md}_+$, we get that
	\[
	\norm{\x_i}^2 + \norm{\x_j}^2 - 2 \x_i^\top [\bG]_{ij} \x_j \geq 0.
	\]
	This gives us
	\[
	\x_i^\top [\bG]_{ij} \x_j \leq \frac{\norm{\x_i}^2 + \norm{\x_j}^2}{2} = 1.
	\]
	Similarly, by replacing $-\x_j$ with $\x_j$ in the $j$-th block of $\x$, we get
	\[
	\x_i^\top [\bG]_{ij} \x_j \geq -1.
	\]
	Putting these together, we have
	\[
	\left| \x_i^\top [\bG]_{ij} \x_j \right| \leq 1.
	\]
	Unit vectors $\x_i$, $\x_j \in \R^{Md}$ were arbitrary, and thus, the result follows.
	
	$(c)$ 
	Consider the spectral decomposition of $\bG$,
	\[
	\bG = \sum_{i=1}^{d} \alpha_i \v_i \v_i^\top
	\]
	where $\alpha_i > 0$ are the non-zero eigenvalues, and $\v_i$ are the corresponding orthonormal eigenvectors. Let
	\[
	\bB = \left[ \sqrt{\alpha_1} \v_1 \cdots \sqrt{\alpha_d} \v_d \right]^\top.
	\]
	Regard $\bB$ as a block-row, where each element is of size $d \times d$. Notationally, $\bB = [\bB_1 \cdots \bB_M]$, where $\bB_i \in \R^{d \times d}$. Thus,
	\[
	[\bG]_{ij} = \bB_i^\top \bB_j.
	\] 
	In particular, since $[\bG]_{ii} = \bI_d$, we get that
	\[
	\bB_i^\top \bB_i = \bI_d.
	\]
	Now, suppose $\v_1 = (\v_{11},\cdots,\v_{1M})$, where $\v_{1j} \in \R^d$. Note that $\sqrt{\alpha_1}\v_{1j}^\top$ forms the first row of $\bB_j$. From orthogonality of $\bB_j$, we get that, for every $j$,
	\[
	\norm{\sqrt{\alpha_1} \v_{1j}}^2 = 1.
	\]
	Or,
	\[
	\norm{\v_{1j}}^2 = \frac{1}{\alpha_1}, \:\: j \in [1:M].
	\]
	Now, since $\norm{\v_1}^2 = 1$, we have 
	\[
	\sum_{j=1}^{M}\norm{\v_{1j}}^2 = 1,
	\]
	from where we get that $\alpha_1 = M$. Similarly, $\alpha_2 = \cdots = \alpha_d = M$.


\subsection{Proof of Proposition \ref{prop:dualform}}
Lemma \ref{lem:kktrel} tells us that $(\bC+\Lam^\ast)\:\bG^\ast = \zer$. That is,
\[
\Lam^\ast \bG^\ast = - \bC \bG^\ast.
\]
Lemma \ref{lem:kktrel} also tells us that $\Lam^\ast$ is a symmetric block-diagonal matrix, and that $[\bG^\ast]_{ii} = \bI_d$. Using these facts, and comparing the diagonal blocks of the left-hand side and the right-hand side, we get
\[
[\Lam^\ast]_{ii} = - [\bC \bG^\ast]_{ii}.
\]
Thus, 
\[
\Lam^\ast = -\mathrm{bd}(\bC \bG^\ast),
\]
where, $\mathrm{bd}$ is the linear operator that leaves the diagonal blocks untouched and sets other elements to $0$.

\subsection{Proof of Proposition \ref{prop:lamblockdiag}}
	We prove the proposition for \ref{ncadmm} by induction. The proof for \ref{cadmm} is exactly the same since the $\bH$-update and $\Lam$-update steps are identical for \ref{ncadmm} and \ref{cadmm}. Clearly, the proposition holds for $k = 0$. Assume that the proposition holds for $k = k_0$. We will show that it then has to hold for $k = k_0 + 1$. Consider the $\bH$-update step in \eqref{ncadmmproj}
	\[
	\bH^{k_0+1} = \Pi_\Theta \left( \bG^{k_0+1} + \rho^{-1} \Lam^{k_0} \right).
	\]
	We know that $\Theta$ is the set of symmetric matrices for which $d \times d$ diagonal blocks are $\bI_d$. It is clear that projection of a matrix on $\Theta$ is obtained by setting the diagonal blocks of the matrix to $\bI_d$. In other words, $\Pi_\Theta(\cdot)$ affects only the diagonal blocks of its argument. By induction hypothesis, $\Lam^{k_0}$ is block diagonal, and thus adding it to $\bG^k_0$ in the $\bH$-update step does not affect the projection on $\Theta$, or, 
	\[
	\bH^{k_0+1} = \Pi_\Theta \left( \bG^{k_0+1} \right).
	\]
	Now, consider the $\Lam$-update step,
	\[
	\Lam^{k_0+1} = \: \Lam^{k_0} + \rho \left( \bG^{k_0+1} - \bH^{k_0+1} \right).
	\]
	Since $\Pi_\Theta(\cdot)$ affects only the diagonal blocks of its arguments, it is clear that the off-diagonal blocks of $\bH^{k_0+1}$ and $\bG^{k_0+1}$ are the same. Thus, $\bG^{k_0+1} - \bH^{k_0+1}$ is a block diagonal matrix, which along with the hypothesis that $\Lam^{k_0}$ is block diagonal, implies that $\Lam^{k_0+1}$ is block diagonal.

\subsection{Proof of Lemma \ref{lem:stable}}
	Since $\mathrm{rank}(\bG_0) = d$ we can write the spectral decomposition of the ground truth solution as (see Proposition \ref{prop:Gprop}(c))
	\[
	\bG_0 = \sum_{i=1}^{d} M \s_i \s_i^\top
	\]
	where $\s_i, \ i \in [1:d]$ are orthonormal. Similarly, we can write the optimal solution corresponding to the perturbed data $\bC$ as
	\[
	\bG^\ast = \sum_{i=1}^{Md} \alpha_i \g_i \g_i^\top
	\]
	where $\g_i, \ i \in [1:d]$ are orthonormal.
	
	\noindent Let $\bP$ be the orthoprojector on $\mathrm{nullspace}(\bC_0)$. Since $\bC_0 \bG_0 = \zer$, $\mathrm{rank}(\bG_0) = d$, and $\mathrm{nullity}(\bC_0) = d$, we deduce that
	\[
	\bP = \sum_{i=1}^{d} \s_i \s_i^\top = \frac{1}{M} \bG_0.
	\]
	Clearly, $\bR = \bI - \bP$ is the orthoprojector on $\mathrm{range}(\bC_0)$, and
	\[
	\g_i = \bP \g_i + \bR \g_i.
	\]
	Let $\h_i := \bR \g_i$. As will be apparent, we want to lowerbound $\dotp{ \bC_0, \bG^\ast } $. Now,
	\begin{equation} \label{stable0}
	\begin{aligned}
	\dotp{ \bC_0, \bG^\ast } 
	&= \sum_{i=1}^{Md} \alpha_i \big( \g_i^\top \bC_0 \g_i \big) \\
	&= \sum_{i=1}^{Md} \alpha_i \big( \h_i^\top \bC_0 \h_i \big) \\
	&\geq \lambda_{d+1}(\bC_0) \sum_{i=1}^{Md} \alpha_i \, \norm{\h_i}^2.
	\end{aligned}
	\end{equation}
	Thus, to get a lowerbound on $\dotp{ \bC_0, \bG^\ast } $, we just have to lowerbound $\sum_{i=1}^{Md} \alpha_i \, \norm{\h_i}^2$. To do this, we first reformulate $\norm{\h_i}^2$ as,
	\[
	\begin{aligned}
	\norm{\h_i}^2 
	&= \norm{\bR \g_i}^2 \\
	&= \g_i^\top \bR \g_i \\
	&= \g_i^\top (\bI - \bP) \g_i \\
	&= \norm{\g_i}^2 - \frac{1}{M} \big( \g_i^\top \bG_0 \g_i \big) \\
	&= 1 - \frac{1}{M} \big( \g_i^\top \bG_0 \g_i \big),
	\end{aligned}
	\]
	where the second equality holds because $\bR^2 = \bR$, since $\bR$ is an orthoprojector.
	Thus,
	\begin{equation} \label{stable1}
	\begin{aligned}	
	\sum_{i=1}^{Md} \alpha_i \norm{\h_i}^2 
	&= \sum_{i=1}^{Md} \alpha_i - \frac{1}{M} \sum_{i=1}^{Md} \alpha_i \big( \g_i^\top \bG_0 \g_i \big) \\
	&= Md - \frac{1}{M}  \dotp{ \bG_0, \bG^\ast } .
	\end{aligned}
	\end{equation}
	Now, to upperbound $ \dotp{ \bG_0, \bG^\ast } $, we note that
	\begin{equation} \label{stable2}
	\begin{aligned}
	\dotp{ \bG_0, \bG^\ast } 
	&= \frac{1}{2} \norm{\bG_0}^2 + \frac{1}{2} \norm{\bG^\ast}^2 - \frac{1}{2} \norm{\Del}^2 \\
	&\leq \frac{1}{2} M^2 d + \frac{1}{2} M^2 d - \frac{1}{2} \norm{\Del}^2 = M^2 d - \frac{1}{2} \norm{\Del}^2
	\end{aligned}
	\end{equation}
	where, to get the first inequality, we use the fact that any singular value of $[\bG^\ast]_{ij}$ is at most $1$ for every $1 \leq i,j \leq M$ (Proposition \ref{prop:Gprop}).
	Combining \eqref{stable1} and \eqref{stable2}, we get
	\[
	\sum_{i=1}^{Md} \alpha_i \norm{\h_i}^2 \geq \frac{1}{2M} \norm{\Del}^2.
	\]
	Plugging this in \eqref{stable0} gives us,
	\[
	\dotp{ \bC_0, \bG^\ast } \geq \frac{\lambda_{d+1}(\bC_0)}{2M} \norm{\Del}^2.
	\]
	Now,
	\[
	\begin{aligned}
	-\norm{\bW} \, \norm{\Del} 
	&\leq \dotp{ \bW, \Del } \\
	&= \big( \dotp{ \bC, \bG^\ast } - \dotp{ \bC, \bG_0 } \big) + \big( \dotp{ \bC_0, \bG_0 } - \dotp{ \bC_0, \bG^\ast } \big) \\
	&\leq - \dotp{ \bC_0, \bG^\ast }
	\end{aligned}
	\]
	where the first inequality is Cauchy-Schwarz (with inner product between two symmetric matrices defined as the trace of their product), and the last inequality is due to the fact that the term in first paranthesis is negative by optimality of $\bG^\ast$, and that the first term in second paranthesis is $0$.
	So,
	\[
	\begin{aligned}
	\norm{\bW} \, \norm{\Del} 
	&\geq \dotp{\bC_0,\bG^\ast} \\
	&\geq \frac{\lambda_{d+1}(\bC_0)}{2M} \norm{\Del}^2
	\end{aligned}
	\]
	which finally gives us,
	\[
	\norm{\Del} \leq \frac{2M}{\lambda_{d+1}(\bC_0)} \norm{\bW}.	
	\]

\subsection{Proof of Lemma \ref{lem:cadmm}}
	The proof of this lemma essentially follows the convergence proof for convex ADMM presented in \cite{boyd}. Note that
	\begin{equation} \label{boyd0}
	\begin{aligned}
	&\norm{\bA^{k} - \bA^*}^2 - \norm{\bA^{k+1} - \bA^*}^2 \\
	&= \norm{\bH^{k} - \bH^*}^2 - \norm{\bH^{k+1} - \bH^*}^2
	+ \rho^{-2} \left(\norm{\Lam^{k} - \Lam^*}^2 - \norm{\Lam^{k+1} - \Lam^*}^2 \right) \\
	\end{aligned}
	\end{equation}
	where we have used the fact that, for every $k$,
	\[
	\dotp{ \bH^k - \bH^\ast, \Lam^k - \Lam^\ast } = 0,
	\]
	because $\bH^k - \bH^\ast$ has $\zer$ on the diagonal blocks, and $\Lam^k - \Lam^\ast$ is block diagonal. 
	Further algebraic manipulations give us
	\begin{equation} \label{boyd3}
	\begin{aligned}
	&\norm{\bA^{k} - \bA^*}^2 - \norm{\bA^{k+1} - \bA^*}^2 \\
	&= \norm{\bH^{k} - \bH^{k+1}}^2 + \norm{\bG^{k+1} - \bH^{k+1}}^2
	- 2\dotp{ \bH^{k} - \rho^{-1}(\bC+\Lam^k) - \bG^{k+1}, \bH^* - \bG^{k+1} } \\
	&\phantom{}- 2 \rho^{-1}\left(\dotp{ \bC,\bH^* - \bG^{k+1} } - \dotp{ \Lam^*,\bG^{k+1} - \bH^{k+1} } \right)
	\end{aligned}	
	\end{equation}
	Observe that the first two terms in \eqref{boyd3} are always nonnegative. Nonnegativity of the third term in \eqref{boyd3} follows from the convex projection property, i.e.,  
	\[
	\dotp{ \bH^{k} - \rho^{-1}(\bC+\Lam^k) - \bG^{k+1}, \bH^* - \bG^{k+1} } \leq 0.
	\]
	To see why the fourth term in \eqref{boyd3} is nonnegative, observe that the Lagrangian for the convex program \ref{csdp} is given by
	\[
	\cL(\bG,\bH,\Lam) = \mathrm{Tr}(\bC \bG) + \dotp{ \Lam, \bG-\bH } .
	\]
	We have already seen in the proof of Lemma \ref{lem:kktrel} that strong duality holds for \ref{csdp}. Thus, from the saddle point property of the Lagrangian at global optimum \cite{boydbook}, we get
	\[
	\cL(\bG^\ast,\bH^\ast,\Lam^\ast) \leq \cL(\bG^{k+1},\bH^{k+1},\Lam^\ast),
	\]
	which gives us
	\[
	\dotp{ \bC,\bH^* - \bG^{k+1} } - \dotp{ \Lam^*,\bG^{k+1} - \bH^{k+1} } \leq 0.
	\]



\bibliographystyle{spbasic}      
\bibliography{references}   

\end{document}